\documentclass[a4paper,10pt]{article}

\usepackage{amsfonts, epsfig, amsmath, amssymb, color, mathabx,amsthm}
\usepackage{mathrsfs}
\usepackage{cite}
\renewcommand{\i}{\mathrm{i}}
\newcommand{\ex}{\mathrm{e}}
\newcommand{\di}{\mathrm{d}}

\newcommand{\cL}{\mathcal{L}}

\newcommand{\bR}{\mathbb{R}}

\usepackage{empheq}
\numberwithin{equation}{section}

\newtheorem{thm}{Theorem}[section]
\newtheorem{cor}[thm]{Corollary}
\newtheorem{lem}[thm]{Lemma}
\newtheorem{prp}[thm]{Proposition}

\theoremstyle{definition}

\newtheorem{exa}[thm]{Example}

\newcommand{\e}{\varepsilon}

\newcommand{\be}{\begin{equation}}
\newcommand{\ee}{\end{equation}}
\newcommand{\ben}{\begin{equation*}}
\newcommand{\een}{\end{equation*}}

\newcommand{\ba}{\begin{equation}\begin{aligned}}
\newcommand{\ea}{\end{aligned}\end{equation}}

\begin{document}

\title{Spectral Analysis of a Discrete Metastable System \\ Driven by L\'{e}vy Flights}
\author{Toralf Burghoff\footnote{Institute of Medical Biometry and Statistics, University of L\"ubeck,
Ratzeburger Allee 160, Geb.\ 24, 23562 L\"ubeck Germany; 
toralf.burghoff@imbs.uni-luebeck.de} \ \
and \
Ilya Pavlyukevich\footnote{Institute for Mathematics, 
Friedrich Schiller University Jena, Ernst--Abbe--Platz 2, 
07743 Jena, Germany; ilya.pavlyukevich@uni-jena.de}}
\date{\null}
\maketitle
\abstract
\noindent
In this paper we consider a finite state time discrete Markov chain that mimics the behaviour of solutions of the
stochastic differential equation
\begin{equation*}
X_t^\varepsilon(x)=x-\int_0^t U^\prime(X_s^\varepsilon)\, \di s+\varepsilon L_t,
\end{equation*}
where $U$ is a multi-well potential with $n\geq 2$ local minima and $L=(L_t)_{t\geq 0}$ is a symmetric $\alpha$-stable
L\'{e}vy process (L\'evy flights process). We investigate the spectrum of the 
generator of this Markov chain in the limit $\varepsilon\rightarrow 0$ and localize the top $n$ eigenvalues 
$\lambda^\varepsilon_1,\ldots,\lambda^\varepsilon_n$. These eigenvalues
turn out to be of the same algebraic order 
$\mathcal O(\varepsilon^\alpha)$
and are well separated 
from the rest of the spectrum by a spectral gap. We also determine the limits 
$\lim_{\varepsilon\to 0}\varepsilon^{-\alpha} \lambda^\varepsilon_i$, $1\leq i\leq n$, 
and show that the corresponding eigenvectors are
approximately
constant over the domains which correspond to the potential wells of $U$.

\medskip

\noindent
\textbf{Keywords:} metastability; L\'evy flights; $\alpha$-stable L\'evy process; eigenvalues; spectral gap; Markov chain; transition times;
fractional Laplacian; semiclassical limit.
\medskip

\noindent
\textbf{AMS Subject Classification:} 60J10$^*$, 60J75, 47A75, 47G20, 15B51, 15A18

\section{Introduction. Spectral properties of Gaussian small noise diffusions}

Let $U\in C^3(\bR,\bR)$ be a one-dimensional multi-well potential (see, e.g.\ Fig.~\ref{f:potential}) such that 
$|U'(x)|>c|x|^{1+c}$ as $x\rightarrow \pm \infty$ for some $c>0$.
Assume that
$U$ has exactly $n$ local minima $\mathfrak{m}_i$, 
$1\leq i\leq n,$ as well as $n-1$ local maxima $\mathfrak{s}_i$, $1\leq i\leq n-1$, enumerated in
increasing order
\begin{align}
-\infty=\mathfrak{s}_0<\mathfrak{m}_1<\mathfrak{s}_1<\mathfrak{m}_2<\cdots<\mathfrak{m}_n<\mathfrak{s}_n=\infty.
\end{align}
Moreover, we suppose that all these extrema are non-degenerate, i.e.\ 
\begin{align}
U''(\mathfrak{m}_i)>0,\ 1\leq i\leq n,\quad \text{and}\quad U''(\mathfrak{s}_i)<0,\ 1\leq i\leq n-1,
\end{align}
and denote
\begin{align}
\mathcal{M}\coloneqq \{\mathfrak{m}_1,\ldots,\mathfrak{m}_n \}.
\end{align}
The points $\mathfrak{m}_i$, $i=1,\dots,n$, are exponentially stable attractors of a deterministic dynamical system generated by 
the ordinary differential equation $\dot x_t=-U'(x_t)$. For any initial position 
$x\in (\mathfrak{s}_{i-1},\mathfrak{s}_i)$, the solution $x_t(x)$ does not leave the domain of attraction 
$(\mathfrak{s}_{i-1},\mathfrak{s}_i)$ and $x_t(x)\to\mathfrak{m}_i$ as $t\to\infty$. 

Although deterministic gradient systems with multiple attractors appear naturally in various application problems 
(e.g.\ climate modelling \cite{BenziPSV-81,Ditlevsen-99a,Nicolis-82}, 
physics \cite{Kramers-40}, 
robotics \cite{byl2009metastable}, economics \cite{Tu-94}),
their evident drawback often consists in an impossibility of transitions between the domains of attraction. 
To overcome this limitation, a small noise is often added to the model. 
Intuitively the perturbed random system stays in the vicinity of one of the attractors $\mathfrak{m}_i$ most of the time, making
sporadic transitions to the other domains of attraction.
This type of behaviour is often referred to as \emph{metastability}.
The essential characteristics of the small noise system as a metastable hopping process are
the life times in the domains of attraction, transition probabilities between the wells, the most probable state etc.
Naturally, these characteristics depend heavily on the type of the noisy perturbation. 
The case of Brownian perturbations is the most well studied.

For a standard Brownian motion $W=(W_t)_{t\geq 0}$, consider the SDE
\begin{align}
\label{SDEW} 
 X_t^\varepsilon(x)=x-\int_0^t U^\prime(X_s^\varepsilon)\,\di s+\varepsilon W_t.
\end{align}
Under the conditions on $U$ formulated above, for any $x\in\bR$, \eqref{SDEW} has a unique 
strong solution $X^\e=(X^\e_t)_{t\geq 0}$ (see e.g.\ \cite{Cerrai-01}).
Since the classical results by   
Kramers \cite{Kramers-40}, it is known that the mean life times of $X^\e$ in the potential wells are exponentially 
long with respect to the parameter
$\e$, more precisely of the order 
$\min\{ \ex^{2(U(\mathfrak{s}_{i- 1})-U(\mathfrak{m}_i))/\e^2 }, \ex^{2(U(\mathfrak{s}_{i+1})-U(\mathfrak{m}_i))/\e^2 }$\} 
for the $i$-th well.
Considering a generic perturbed gradient system on
exponentially long time scales of the order $T(\e)=\ex^{\mu/\e^2}$, $\mu>0$, one can recover different behaviours.
In particular on the slowly growing (short) scales, 
when $\ex^{\mu/\e^2}$ is less than the life times in all the wells ($\mu<2 (U(\mathfrak{s}_{i\pm 1})-U(\mathfrak{m}_i) )$, $i=1,\dots n$),
the trajectory $(X^\e_{T(\e)\cdot t})_{t\in[0,1]}$ does not succeed in leaving any of the domains of attraction. 
On longer time scales, the trajectory may leave
some of the domains of attraction thus making transitions and eventually forming a sublimit distribution. On the longest scale, 
$X^\e$ converges to its stationary law which is concentrated in the vicinity of the potential's global minimum. 
For a generic potential, there are $n-1$ such critical values which are different and can be ordered in decreasing order
$\mu_2>\mu_3>\dots>\mu_{n}$. These values determine the time scales $T(\e)=\ex^{\mu_i/\e^2}$ which
distinguish the sublimit distributions.
The values $\mu_i$, $i=2,\dots, n$, can be calculated probabilistically in terms of the heights of 
the potential barriers of $U$ with the help of
the large deviations theory by
Freidlin and Wentzell (see \cite[\S VI.6]{FreidlinW-98} or a recent exposition \cite{Ca13}).

Recall that $X^\e$ is a strong Markov processes with the infinitesimal generator
\begin{equation}
\cL_\e f(x)=\frac{\e^2}{2}f''(x)-U'(x)f'(x),\quad f\in C^\infty_0(\bR,\bR).
\end{equation}
The existence of multiple sublimit distributions is reflected in the spectral properties of the operator $\mathcal L_\e$.
The generator $\cL_\e$ is a negative definite essentially self-adjoint operator in the weighted space 
$L^2(\bR,\ex^{-2U(x)/\e^2}\,\di x)$. It 
has a discrete non-positive spectrum $0\geq-\lambda_1^\e\geq -\lambda_2^\e\geq \cdots$ 
and a corresponding sequence of eigenfunctions $\{\Phi^\e_i\}_{i\geq 1}$ 
which form an orthonormal basis (see \cite[Chapter 3]{HIPP2014} for the detailed exposition).
Whereas its top eigenvalue is identically zero, $\lambda^1_\e\equiv 0$, with the corresponding eigenfunction being a constant, 
it turns out that the next $n-1$ eigenvalues are of the order $\ex^{-\mu_i/\e^2}$
where the rates $\mu_i$ are obtained above.
The principle non-zero eigenvalue $\lambda^2_\e$ corresponding to the longest time scale $\ex^{\mu_2/\e^2}$ determines the
speed of convergence of the law of $X_t^\e$
to the stationary distribution as $t\to\infty$; further eigenvalues recover the speeds of convergence to sublimit distributions.
With the help of the Fourier method of separation of variables, the functional $u_\e(t,x)= \mathbb E_x f(X^\e_t)$ 
which is the solution of the Cauchy problem $\partial_t u=\cL_\e u$, $u_\e(0,x)=f(x)$, 
can be represented as 
\begin{equation}
u_\e(t,x)= 
\frac{\displaystyle\int_\bR f(y) \ex^{-2U(y)/\e^2}\,\di y}{\displaystyle\int_\bR \ex^{-2U(y)/\e^2}\,\di y   }    +  
\sum_{i=2}^{n} \Big(
\int_\bR f(y) \Phi_i^\e(y)\ex^{-2U(y)/\e^2}\,\di y\Big)\cdot \Phi_i^\e(x)\ex^{-\lambda_i^\e t} + R_\e(t,x).
\end{equation}
The remainder term $R_\e$ can be shown to be negligible compared to the leading $n$ terms of the expansion due to the 
spectral gap property, namely, that there is a constant $M$ not depending on $\e$ such that $\lambda^\e_{n+1}\geq M$ for all $\e\in(0,1]$.
Moreover, the eigenfunctions $\Phi_i^\e$, $2\leq i\leq n$, are almost constant over the potential wells. 

A rich literature is devoted to the analysis of metastability of the system of the type \eqref{SDEW} 
on different levels of rigour. A probabilistic characterization
of the principle eigenvalue of $\mathcal L_\e$ in a bounded domain with the Dirichlet boundary conditions
was obtained by Khasminski \cite{Khasm-59} (see also Friedman \cite[Lemma 1.1]{Friedman-73}).
Schuss \textit{et al.} \cite{MatkowskyS-79,MatkowskyS-81,Schuss-10} determined
the shape of the eigenfunction $\Phi^\e_2$ and the 
asymptotics of the eigenvalue $\lambda^\e_2$, especially its subexponential prefactor in terms of the values $U''(\mathfrak{s}_i)$ and 
$U''(\mathfrak{m}_i)$ with the help of formal asymptotic expansions
of solutions of second order ordinary differential equations.  
Buslov, Makarov and Kolokoltsov \cite{Makarov-85,BuslovM-88,BuslovM-92,KolokoltsovM-96,Kolokoltsov-00} 
established a connection between the top spectrum of the diffusion's generator and the 
spectrum of the matrix of the inverse mean life times of $X^\e$ in the potential wells, found very exact approximations for the 
corresponding eigenfunctions and proved the existence of an $\e$-independent spectral gap. 
Bovier, Eckhoff, Gayrard and Klein in \cite{Eckhoff00,BovEckGayKle02,BovierEGK-04,BovierGK-05} developed a 
potential theoretic approach to metastability of Markov chains and gradient diffusions and obtained very exact asymptotics
of the top eigenvalues.
Metastable behaviour of diffusions in a double-well potential was studied 
by probabilistic methods in \cite{KipnisN-85,GalvesOV-87,OlivieriV-03}.
Applications of the spectral theory to simulated annealing can be found in 
\cite{ChiangHS-87,HolleyKS-89,HwangS-90}, to stochastic resonance in \cite{HIPP2014}, and to analysis of 
molecular dynamics \cite{schutte2013metastability,MetSchVdE-09}.
Berglund and Gentz \cite{BerGen-10} studied the behaviour of small noise diffusions in potentials with non-quadratic extrema.
We refer the reader to a recent review \cite{Berglund-13} by Berglund for 
further references on the subject.

In recent time, equations driven by non-Gaussian noise, especially $\alpha$-stable L\'evy processes (L\'evy flights), 
are being adopted for description and modelling of various real world phenomena (see e.g.\ \cite{MetCheKla12}).
Let $L=(L_t)_{t\geq 0}$ be a symmetric $\alpha$-stable L\'evy process with the characteristic function 
\begin{equation}
\label{eq:chf}
\mathbb E \ex^{\i u L_t}=\ex^{-c(\alpha) t |u|^\alpha},\ u\in\bR, \ \alpha\in(0,2),
\end{equation}
where $c(\alpha)=\alpha \int_0^\infty \frac{1-\cos y}{y^{1+\alpha}}\,\di y\in(0,\infty)$.
Such a process has heavy tails and for convenience the constant $c(\alpha)$ is chosen to guarantee the following tail asymptotics:
\begin{equation}
\lim_{u\to\infty }u^\alpha\mathbb P(|L_1|\geq u)=1.
\end{equation} 
For $\e>0$, the SDE 
\begin{align}\label{SDE} 
 X_t^\varepsilon(x)=x-\int_0^t U^\prime(X_s^\varepsilon)\,\di s+\varepsilon L_t
\end{align}
possesses a unique strong solution $X^\e=(X^\e_t)_{t\geq 0}$ which is a strong Markov processes (see \cite{Applebaum-09,ImkellerP-08}).

In the limit $\e\to 0$, the process $X^\e$ enjoys the metastable behaviour in the following sense.
\begin{thm}[Theorem 1.1, \cite{ImkellerP-08}]
\label{th:1}
Let $k=1,\dots,n$ and
$x\in (\mathfrak{s}_{k-1},\mathfrak{s}_{k-1})$.
Then the process $(X^\varepsilon_{t/\varepsilon^\alpha})_{t> 0}$, $X^\e_0=x$, converges, 
as $\varepsilon\rightarrow 0$, in the sense of finite dimensional distributions to a Markov chain
$Y=(Y_t)_{t> 0}$, $Y_0=\mathfrak{m}_k$, on the state space
$\mathcal{M}$ with the stable conservative generator $\mathbf{Q}=(q_{i,j})_{i,j=1}^n$ given by
\begin{align}\label{eintraegeQ}
   q_{i,j}=\begin{cases}\displaystyle
         \frac{1}{2}\Big|\frac{1}{|\mathfrak{s}_{j-1}-\mathfrak{m}_i|^\alpha}-\frac{1}{|\mathfrak{s}_{j}-\mathfrak{m}_i|^\alpha}\Big|, &i\neq j, \\
  \displaystyle\vphantom{\Bigg|}
  - \frac{1}{2} \Big(\frac{1}{|\mathfrak{s}_{i-1}-\mathfrak{m}_i|^\alpha}+\frac{1}{|\mathfrak{s}_{i}-\mathfrak{m}_i|^\alpha}\Big), & i=j.
           \end{cases}
\end{align}
\hfill $\blacksquare$
\end{thm}
In other words, there is a time scale $T(\e)=\e^{-\alpha}$ on which the process $X^\e$ reminds of a finite state Markov chain $Y$. 
It was also shown in \cite{ImkellerP-08}, that on slower time scales $\e^{-\mu}$, $\mu<\alpha$, the process $X^\e$ 
does not leave the potential well where it has started. On faster time scales $\e^{-\mu}$, $\mu>\alpha$, one cannot obtain a 
meaningful limit of $X^\e$ since its life times in the potential wells converge to zero and the process 
$X^\varepsilon_{t/\varepsilon^\alpha}$ persistently jumps between different wells.

The generator of $X^\e$ is the integro-differential operator
\begin{equation}
\begin{aligned} 
\mathcal D_\e f(x)&=  \frac{\alpha}{2}\e^\alpha\int_0^\infty \frac{f(x+z)-2f(x)+f(x-z)}{z^{1+\alpha}}\,\di z -U'(x)f'(x)\\
&= - \frac{\alpha}{2}\e^\alpha(-\Delta)^{\alpha/2}f(x)-U'(x)f'(x),\quad f\in C_0^\infty (\mathbb R,\mathbb R),
\end{aligned}
\end{equation}
and the metastability result of Theorem \ref{th:1} suggests that the top spectrum of $\mathcal D_\e$ should remind of the spectrum of the 
Markov chain $Y$. One can expect that the top $n$ eigenvalues 
$\lambda^\varepsilon_1=0,\lambda^\varepsilon_2,\ldots,\lambda^\varepsilon_n$ of $\mathcal D_\e$
should be closely connected with the eigenvalues of the matrix $\mathbf{Q}$ and well separated from the rest of the spectrum 
 by a spectral gap. Since the time scale in the metastability result 
is unique, the eigenvalues should be all
of the same 
order $\mathcal O(\varepsilon^{\alpha})$.

Unfortunately not much can be said about the spectrum of $\mathcal D_\e$. No weighted space is known where 
$\mathcal D_\e$ is self-adjoint and no general spectral theory can be applied in this case. 
Although it is known that the invariant distribution exists and is unique \cite{SamorodnitskyG-03,Kulik-09}, explicit 
formulae for its density can be obtained only in a very 
few particular cases (mainly for the Cauchy process $\alpha=1$ and polynomial potentials, see \cite{ChechkinGKM-04,DubSpa07,DybSokChe10})

In this paper we make the first step towards a better understanding of the 
spectral properties of $X^\e$ by reduction of a jump-diffusion to a finite state discrete time Markov 
chain and analysing its spectral properties in the limit $\e\to 0$.


\section{A discrete L\'evy driven system and the main result}

We construct a discrete-time Markov chain on a finite state space
that mimic the metastable behaviour of the solution $X^\varepsilon$ of \eqref{SDE} in the limit of small $\e$.
The construction is based on the standard Euler scheme for SDEs. 

For any $\gamma>0$ there are  a range parameter $R>0$, a time step $h>0$ as well as the spacial mesh parameter 
$\delta>0$ such that $R^{-1}+h+\delta\leq \gamma$ and such that the following holds true. 

All the local minima of $U$ belong to $[-R,R)$, $-R<\mathfrak{m}_1<\mathfrak{m}_n<R$. With $R$ fixed, let us redefine the 
points $\mathfrak{s}_0:=-R$ and $\mathfrak{s}_n:=R$ for convenience.

There exists a finite set of points $\mathcal S$ and a partition
$\bigsqcup_{x\in\mathcal S} I_x=[\mathfrak{s}_0,\mathfrak{s}_n)=$ consisting of the intervals $I_x=[a_x,b_x)$
such that $\max_{x\in \mathcal S}|b_x-a_x|\leq \delta/2$, each interval  $I_x$ contains only one point from $\mathcal S$ and 
$x\in (a_x,b_x)$, $\mathcal M\subset \mathcal S$,
$\{\mathfrak{s}_0,\dots, \mathfrak{s}_n\}\subset \{a_x,b_x\}_{x\in\mathcal S}$,
and for any $x\in \mathcal S\backslash\mathcal M$
\begin{equation}
\label{eq:1}
x-hU'(x)\notin I_x.
\end{equation}

\begin{exa}
The set $\mathcal S$ and the intervals $I_x$ can be constructed as follows. Let $R>\frac{3}{\gamma}\vee |\mathfrak{m}_1|\vee |\mathfrak{m}_n|$.
Let $\delta<\frac{\gamma}{3}\wedge \frac{1}{8}\min_{i,j} |\mathfrak{s}_i-\mathfrak{m}_j|$. 

First, we demand that all $\mathfrak{m_i}\in\mathcal S$ and set
$I_{\mathfrak{m}_i}=[\mathfrak{m}_i-\frac{\delta}{4},\mathfrak{m}_i+\frac{\delta}{4})$.
Choose  $0<h\leq \gamma/3$ 
such that for all $i=1,\dots, n$, all $x\in [\mathfrak{s}_{i-1}+ \frac{\delta}{4}, \mathfrak{m}_i- \frac{\delta}{4}]$
\begin{equation}
x-U'(x)h \in [\mathfrak{s}_{i-1},\mathfrak{m}_i]
\end{equation} 
and 
for $x\in [\mathfrak{m}_{i}+ \frac{\delta}{4} , \mathfrak{s}_i-\frac{\delta}{4}]$ 
\begin{equation}
x-U'(x)h \in [\mathfrak{m}_{i},\mathfrak{s}_i].
\end{equation}

For definiteness, let us construct the partition of the
interval $[\mathfrak{m}_1+\frac{\delta}{4},\mathfrak{s}_1)$. 
Let $z_1$ be the solution of the equation $z-U'(z)h=\mathfrak{m}_1+\frac{\delta}{4}$, $z_1>\mathfrak{m}_1+\frac{\delta}{4}$.
Decompose the interval
\begin{equation}
{}[\mathfrak{m}_1+\frac{\delta}{4}, z_1\wedge \mathfrak{s}_i-\frac{\delta}{4})
\end{equation}
into a finite disjoint union of the intervals $[a_x, b_x)$, $b_x - a_x<\frac{\delta}{2}$, denote by $x$ the middle point of each such 
an interval, and add $x$ to $\mathcal S$. Then due to the fact that $U'(x)>0$ on 
$[\mathfrak{m}_{1}+ \frac{\delta}{4} , \mathfrak{s}_1-\frac{\delta}{4}]$, we get $x-U'(x)h\in I_{\mathfrak{m}_1}$.
Denote $z_2$ the solution of the equation $z-U'(z)h=z_1$, $z_2>z_1$ and 
again decompose the interval $[z_1,z_2 \wedge \mathfrak{s}_1-\frac{\delta}{4})$ a finite disjoint union of the intervals
of the maximal length $\frac{\delta}{2}$. Continue this procedure  and assign the last interval to be  $(a_x,\mathfrak{s}_i]$. 
\hfill $\blacksquare$
\end{exa}
For any $x\in \mathcal S$ let $y^*(x)\in \mathcal S$ be the unique element such that
\begin{equation}
x-hU'(x)\in I_{y^*(x)} .
\end{equation}
Let us denote by $T\colon \mathcal{S}\rightarrow \mathcal{S}$ the mapping corresponding to the operation ``\,*\,'',
i.e.\ $Tx\coloneqq y^*(x)$.
The condition \eqref{eq:1} implies that $Tx=y^*(x)\neq x$ for $x\in\mathcal S\backslash\mathcal M$ whereas
$Tx=y^*(x)= x$ for $x\in\mathcal M$.

On the set $\mathcal S$ define a discrete time deterministic motion $Z^0=(Z^0_k)_{k\geq 0}$ such that
\begin{equation}
\label{eq:2}
Z^0_0=x,\ Z^0_k=y^*(Z^0_{k-1})=T^k x,\ k\geq 1.
\end{equation}
For any $x\in \mathcal S$ the sequence $\{Z^0_k(x)\}_{k\geq 0}$ is monotone.

The deterministic motion $Z^0$ mimics the behaviour of the solutions of the deterministic 
ordinary differential equation $\dot x_t=-U'(x_t)$.
It can also be described as a discrete time Markov chain on $\mathcal S$ with the matrix 
$\mathbf{P}^0=(p^0_{x,y})_{x,y\in\mathcal{S}}$ of one step transition probabilities given by
\begin{align}
p^0_{x,y}\coloneqq\begin{cases}
1, & y=y^*(x),\\
0, &\text{otherwise}.
\end{cases}
\end{align}

Eventually let us construct a discrete-time Markov chain 
$Z^\varepsilon=(Z^\varepsilon_k)_{k\geq 0}$ on the state space $\mathcal{S}$ which mimics the behaviour of 
the jump-diffusion
$X^\e$ and can be considered as
a random perturbation of $Z^0$. 

Denote $\tilde{y}\coloneqq \min\{y\colon y\in\mathcal{S}\}$ and $\hat{y}\coloneqq \max\{y\colon y\in\mathcal{S}\}$ and 
define the  matrix of one-step transition probabilities $\mathbf{P}^\varepsilon=(p^\varepsilon_{x,y})_{x,y\in\mathcal{S}}$ of $Z^\e$ as
\begin{align}\label{deftransprobs1}
 p^\varepsilon_{x,y}\coloneqq \mathbb{P}(x-h U^\prime(x)+\varepsilon h^{\frac{1}{\alpha}}L_1\in I_y),\quad y\neq\tilde{y},\hat{y},
\end{align}
as well as
\begin{align}\label{deftransprobs2}
 p^\varepsilon_{x,\tilde{y}}\coloneqq \mathbb{P}(x-hU^\prime(x)+\varepsilon h^{\frac{1}{\alpha}}L_1\leq b_{\tilde{y}})
\end{align}
\begin{align}\label{deftransprobs3}
p^\varepsilon_{x,\hat{y}}\coloneqq \mathbb{P}(x-hU^\prime(x_k)+\varepsilon h^{\frac{1}{\alpha}}L_1\geq a_{\hat{y}}).
\end{align}
By construction, the Markov chain $Z^\e$ gets reflected at the barriers $-R$ and $R$
which mimics the very fast return from infinity of the trajectory of $X^\e$ due to the fast increase of the potential at infinity.

First we have to study the metastable behaviour of $Z^\e$ as $\e\to 0$. Since $Z^\e$ is constructed in such a way that it
resembles the L\'evy driven jump diffusion $X^\e$ one could expect that the $Z^\e$ demonstrates the same metastable behaviour. Indeed, the
following analogue of the Theorem \ref{th:1} holds true.

\begin{thm}
\label{th:2}
Let $k=1,\dots,n$ and
$x\in \mathcal S\cap [\mathfrak{s}_{k-1},\mathfrak{s}_{k} ]$.
Then the process $(Z^\e_{\lceil \frac{t}{h\varepsilon^{\alpha}}\rceil})_{t> 0}$, $Z^\e_0=x$, converges, 
as $\varepsilon\rightarrow 0$, in the sense of finite dimensional distributions to a Markov chain
$Y=(Y_t)_{t> 0}$, $Y_0=\mathfrak{m}_k$, on the state space
$\mathcal{M}$ with the generator $\mathbf{Q}$ defined in \eqref{eintraegeQ}.
\end{thm}
\begin{proof}
The proof of this result essentially follows the arguments presented in \cite{ImkellerP-08} for the case of a one dimensional jump
diffusion $X^\e$. The arguments are even easier since by construction, $\mathcal S$ does not contain the local maxima of $U$
Furthermore thanks to the reflection condition at $\pm R$ and the finiteness of $\mathcal S$, 
we do not have to consider returns of $Z^\e$ from infinity.
Note that this convergence result does not depend on the size or the particular choice of the state space $\mathcal S$.
For a detailed proof we refer the reader to Chapter 3 in \cite{Burghoff14}.
\end{proof}

The Theorem \ref{th:2} suggests that 
top spectrum of the generator of the Markov chain 
$Z^\varepsilon$
should be closely related to the spectrum of the matrix $\mathbf{Q}$. Let $N=|\mathcal S|$ denote the cardinality of the set $\mathcal S$. 
To compare the generators of $Z^\varepsilon$ and $Y$ we note that
\begin{equation}
\mathbf{P}^\varepsilon=\ex^{\mathbf{P}^\varepsilon-\mathbf{I}_N},
\end{equation}
so that the matrix $\mathbf{P}^\varepsilon-\mathbf{I}_N$
can be viewed as the discrete analogue of an infinitesimal generator.
Taking into account the  time scaling $t\mapsto \frac{t}{h \e^\alpha}$ which appears in the metastability result 
(Theorem \ref{th:2}), we introduce the $N\times N$ matrix
\begin{align}
\label{discretegenerator}
 \mathbf{Q}^\varepsilon\coloneqq \frac{1}{h\varepsilon^\alpha}\left(\mathbf{P}^\varepsilon-\mathbf{I}_{N}\right)
\end{align}
and denote by $\sigma(\mathbf{Q}^\varepsilon)\coloneqq\{\lambda^\varepsilon_1,\ldots,\lambda^\varepsilon_N\}$ and
 $\sigma(\mathbf{Q})\coloneqq\{\lambda^{\mathbf{Q}}_1,\ldots,\lambda^{\mathbf{Q}}_n\}$ the spectra of 
 $\mathbf{Q}^\varepsilon$ and $\mathbf{Q}$, respectively.
 Then the following main theorem holds.
\begin{thm}
\label{theoremeigenwerte}
The spectrum $\sigma(\mathbf{Q}^\varepsilon)$ can be divided into two disjoint parts 
$\sigma_1(\mathbf{Q}^\varepsilon)$ and $\sigma_2(\mathbf{Q}^\varepsilon)$ for which the
following assertions hold:
\begin{itemize}
 \item[(i)] $\sigma_1(\mathbf{Q}^\varepsilon)$ contains precisely $n$ eigenvalues 
 $\lambda^\varepsilon_1,\ldots,\lambda^\varepsilon_n$ such that $\lambda_1^\varepsilon=\lambda^{\mathbf{Q}}_1=0$ and
    \begin{align}
     \lim_{\varepsilon\rightarrow 0}\lambda^\varepsilon_i=\lambda^\mathbf{Q}_i,\quad 2\leq i\leq n;
    \end{align}
\item[(ii)] For any sequence $\{r_\e\}_{\e>0}$ such that $r_\e\to+\infty$, $r_\e \e^{\frac{\alpha}{n+1}}\to 0$ 
there is $\e_0>0$ such that for all $\e\in(0,\e_0]$ and all $i=n+1,\dots, N$
    \begin{align}
| \lambda^\varepsilon_i|\geq r_\e.
    \end{align}
\end{itemize}
\end{thm}

The proof the Theorem \ref{theoremeigenwerte}
does not take unto account any a priori information about the eigenvalues of $\mathbf{Q}$ 
other than the fact that one of them is $0$. However,
we strongly believe that in the generic case, e.g.\ when the potential $U$ does not have any intrinsic symmetries, 
all eigenvalues of $\mathbf{Q}$ are real and simple. This conjecture is supported by numerous computer simulations of 
spectra of a matrix $\mathbf{Q}$,
when the entries of the form \eqref{eintraegeQ} are randomly generated.

Denote by $\boldsymbol{\delta}^j=(\delta^j_x)_{x\in\mathcal{S}}$, $1\leq j\leq n,$ the  indicator function (vector) of the set
$\mathcal{S}_j:=\mathcal{S}\cap[\mathfrak{s}_{i-1},\mathfrak{s}_{i}]$,
i.e.\ $\delta^j_x=1$ for $x\in\mathcal{S}_j$ and $0$ elsewhere.

\begin{thm}\label{Theoremeigenvektoren}
Assume that the eigenvalues of $\mathbf{Q}$ are real and simple.
Let $1\leq i\leq n$ and let $\boldsymbol{\psi}^{\varepsilon,i}=(\psi^{\varepsilon,i}_x)_{x\in\mathcal{S}}$ be the 
right eigenvector of $\mathbf{Q}^\varepsilon$ associated with $\lambda^{\varepsilon}_i$
and normalized such that $\max_{1\leq j\leq n}|\psi^{\varepsilon,i}_{\mathfrak{m}_j}|=1$.
Let $\boldsymbol{\psi}^{\mathbf{Q},i}=(\psi^{\mathbf{Q},i}_j)_{j=1}^n$ be the 
right eigenvector of $\mathbf{Q}$ associated with $\lambda^\mathbf{Q}_i$ and normalized such that
$\max_{1\leq j\leq n}|\psi^{\mathbf{Q},i}_j|=1$.
Then the following limit holds:
\begin{align}
\max_{x\in\mathcal{S}}
\Big|\psi^{\varepsilon,i}_x -\sum_{j=1}^n \psi^{\mathbf{Q},i}_j \delta^j_x\Big|\rightarrow 0,\quad \varepsilon\rightarrow 0.
\end{align}
In particular, for $i=1$, $\lambda^\e_1=0$ and $\boldsymbol{\psi}^{\varepsilon,1}=\boldsymbol{1}$.
\end{thm} 

\newpage
{
\begin{figure}[!ht]
\centerline{
 \input{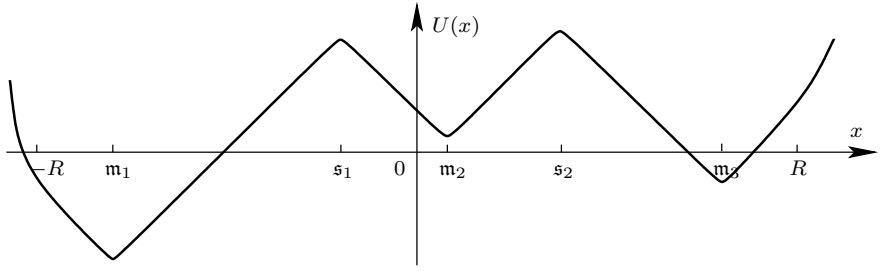} }
\caption{A potential $U$ with the minima $\mathfrak{m}_1=-4.015$, $\mathfrak{m}_2=0.468$, $\mathfrak{m}_3=3.966$
and the maxima $\mathfrak{s}_1=-1.034$, $\mathfrak{s}_2=1.921$; the range $R=5$.
\label{f:potential}} 
\end{figure}
\begin{figure}[!ht]
\centerline{ \input{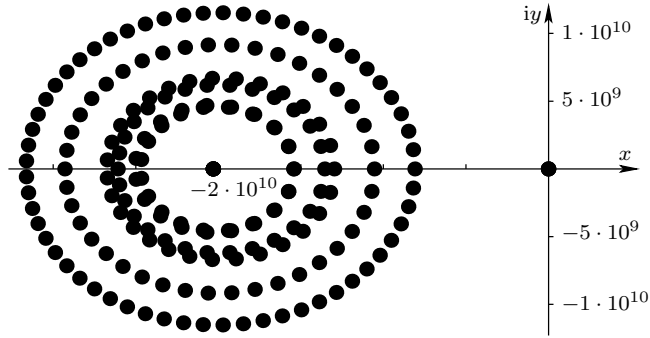} }
\caption{The spectrum of the matrix $\mathbf{Q}^\varepsilon$ on the complex plane. The top three eigenvalues are located in the neighbourhood
of the origin and well separated from the rest of the spectrum.
\label{f:evalues}} 
\end{figure}
\begin{figure}[!ht]
 \centerline{\input{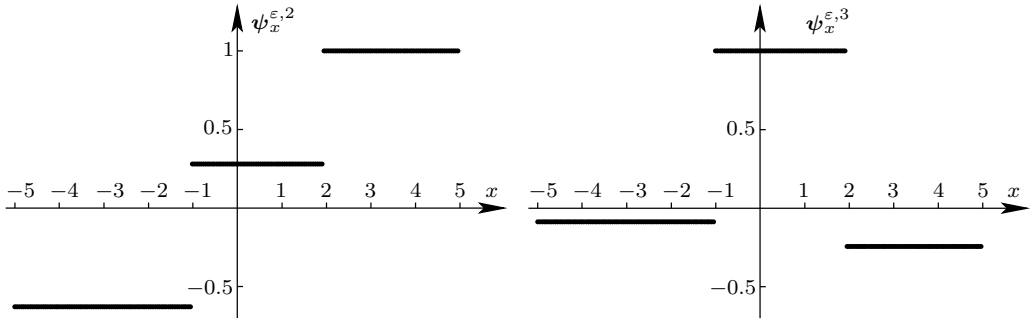} }
\caption{The eigenvectors $\boldsymbol{\psi}^{\e,2}$ and $\boldsymbol{\psi}^{\e,3}$ of the matrix 
$\mathbf{Q}^\varepsilon$ are almost constant over the domains of attraction of $U$.
\label{f:evectors}} 
\end{figure}}
\begin{exa}
Consider a three-well potential $U$ such that, for simplicity of simulations, 
$U'=\pm 1$ outside of small neighbourhoods of the extrema, see Fig.\ \ref{f:potential}, and fix $\alpha=1.8$. The matrix $\mathbf{Q}$ 
calculated according to \eqref{eintraegeQ} has the eigenvalues $\lambda^\mathbf{Q}_1=0$,
$\lambda^\mathbf{Q}_2=-0.124$,
$\lambda^\mathbf{Q}_3=-0.579$,
and the corresponding eigenvectors $\boldsymbol{\psi}^{\mathbf{Q},1}=(1,1,1)$,
$\boldsymbol{\psi}^{\mathbf{Q},2}=(-0.629, 0.280, 1)$,
$\boldsymbol{\psi}^{\mathbf{Q},3}=(-0.088, 1, -0.245)$. We construct the state space $\mathcal S$ consisting of $N=203$ points from the
interval $[-5,5)$ and a Markov chain with the matrix of transition probabilities defined according to \eqref{deftransprobs1},   
\eqref{deftransprobs2},  \eqref{deftransprobs3} with $h=10/N$ and $\varepsilon = 10^{-5}$. The spectrum of the matrix
$\mathbf{Q}^\varepsilon=\frac{1}{\varepsilon^\alpha h}( \mathbf{P}^\varepsilon-\mathbf{I}_N)$ consists of 203 eigenvalues, see
Fig.\ \ref{f:evalues}, and the top eigenvalues are $\lambda^\e_1$, $\lambda^\e_2$ and $\lambda^\e_3$,
coincide with the spectrum of $\mathbf{Q}$ with a four decimal places precision.
The eigenvector $\boldsymbol{\psi}^{\e,1}$ is constant on $\mathcal S$ whereas 
the eigenvectors $\boldsymbol{\psi}^{\e,2}$ and $\boldsymbol{\psi}^{\e,3}$ converge to the values 
$\boldsymbol{\psi}^{\mathbf{Q},2}_k$, $\boldsymbol{\psi}^{\mathbf{Q},3}_k$, $k=1,2,3$, on the parts of $\mathcal S$ corresponding to the 
domains of attraction of $U$, see Fig.\ \ref{f:evectors}.
\end{exa}

\newpage

 
\section{Eigenvalues. Proof of the Theorem \ref{theoremeigenwerte}}

Let $L_1$ by a symmetric $\alpha$-stable random variable with the characteristic function \eqref{eq:chf}. It is well-known 
(see \cite{Berg1952}) that $L_1$
has a density $p_\alpha(x)$ on $\mathbb R$ which can be expanded as $x\to+\infty$ as:
\begin{align}
p_{\alpha}(x)= \frac{\alpha}{2 x^{\alpha+1}}+\mathcal O\Big(\frac{1}{x^{2\alpha+1}}\Big).
\end{align}
In particular, for any fixed $0<a<b$ and $h>0$ there is $C>0$ such that in the limit $\e\to 0$
\begin{align}
\label{eq:exp}
\Big| \mathbb{P}(a\leq \e h^{\frac{1}{\alpha}}L_1 < b)
-\frac{\e^\alpha h}{2}\Big(\frac{1}{a^\alpha}-\frac{1}{b^\alpha}\Big)\Big|
\leq C\e^{2\alpha}. 
\end{align}
The next Lemma follows directly form the formulae \eqref{deftransprobs1}, 
\eqref{deftransprobs2} and  \eqref{deftransprobs3}.

\begin{lem}
\label{l:proptransitionprobabilities}
Let $\alpha\in(0,2)$ and the set $\mathcal S$ and the partition $[-R,R)=\bigsqcup_{x\in\mathcal S}[a_x,b_y)$ be fixed. 
Then there is $\e_0>0$ small enough and a constant 
$C=C(h,\delta,R,\alpha)>0$ such that for all $0<\varepsilon<\varepsilon_0$ the following estimates hold:
\begin{itemize}
 \item [(i)] for any $x\in\mathcal{S}$ and any $y\in\mathcal S$ such that $y\neq y^*(x)$ we have
    \begin{align}\label{estmatrix1}
      \left\vert p^\varepsilon_{x,y}-d_{x,y}\,\varepsilon^\alpha h\right\vert\leq C\varepsilon^{2\alpha}
    \end{align}
    where 
    \begin{align}
      d_{x,y}=
 \begin{cases}
   \displaystyle \frac{1}{2}\Big|\frac{1}{| a_{y}-x+hU'(x)|^\alpha}-\frac{1}{| b_{y}-x+hU'(x)|^\alpha}\Big|, &  y\notin \{\tilde{y}, \hat{y} , y^*(x)\},\\
   \displaystyle
               \frac{1}{2|b_{\tilde y}-x+hU'(x)|^{\alpha}}, & y=\tilde{y},\\
  \displaystyle             
  \frac{1}{2| a_{\hat y}-x+hU'(x)|^{\alpha}}, & y=\hat{y};
 \end{cases}
    \end{align}
 \item[(ii)] for $x\in\mathcal{S}$ and $y=y^*(x)$ we have
    \begin{align}\label{estmatrix2}
     | p^\varepsilon_{x,y^*(x)}- (1-f_{x,y^*(x)}\,\varepsilon^\alpha h)|\leq C\varepsilon^{2\alpha},
    \end{align}
    where 
     \begin{align}
     f_{x,y^*(x)}= \frac{1}{2}\Big(\frac{1}{| a_{y^*(x)}-x+hU'(x)|^\alpha}+\frac{1}{| b_{y^*(x)}-x+hU'(x)|^\alpha}\Big).
    \end{align}
\end{itemize}
\end{lem}
\noindent

\subsection{Proof of Theorem \ref{theoremeigenwerte}}
In the first step we will analyse the structure of the matrix
\begin{align}\label{eigenwertmatrix}
 \mathbf{A}^\varepsilon(\lambda)\coloneqq \mathbf{Q}^\varepsilon-\lambda\mathbf{I}_{N}
=\frac{1}{\varepsilon^\alpha h}\left(\mathbf{P}^\varepsilon-\mathbf{I}_{N}\right)-\lambda\mathbf{I}_{N} 
 ,\quad \lambda\in\mathbb{C},
\end{align}
which obviously plays the essential role in the investigation of the spectrum of $\mathbf{Q}^\varepsilon$.

First we observe that the matrices $\mathbf{P}^\varepsilon$ and $\mathbf{A}^\varepsilon(\lambda)$ possess a block structure. 
Indeed, recalling the decomposition of the state space $\mathcal{S}=\mathcal{S}_1\sqcup\ldots\sqcup\mathcal{S}_n$,
$\mathcal S_i=\mathcal S\cap [\mathfrak{s}_{i-1}, \mathfrak{s}_i]$, $i=1,\dots, n$,
as a disjoint union of states belonging to different potential wells, we may write
\begin{align}
\mathbf{P}^\varepsilon=
 \begin{pmatrix}
  \mathbf{P}^\varepsilon_{\mathcal{S}_1,\mathcal{S}_1} &\vline	 &\cdots	 &\vline	 & \mathbf{P}^\varepsilon_{\mathcal{S}_1,\mathcal{S}_n}\\\hline
  \vdots 			 &\vline	 &	         &\vline  	 &\vdots\\\hline
  \mathbf{P}^\varepsilon_{\mathcal{S}_n,\mathcal{S}_1} &\vline 	 &\cdots 	 &\vline	 & \mathbf{P}^\varepsilon_{\mathcal{S}_n,\mathcal{S}_n}
  \end{pmatrix},
\quad 
\mathbf{A}^\varepsilon(\lambda)=
 \begin{pmatrix}
  \mathbf{A}^\varepsilon_{\mathcal{S}_1,\mathcal{S}_1}(\lambda) &\vline	 &\cdots	 &\vline	 & \mathbf{A}^\varepsilon_{\mathcal{S}_1,\mathcal{S}_n}\\\hline
  \vdots 			 &\vline	 &	         &\vline  	 &\vdots\\\hline
  \mathbf{A}^\varepsilon_{\mathcal{S}_n,\mathcal{S}_1} &\vline 	 &\cdots 	 &\vline	 & \mathbf{A}^\varepsilon_{\mathcal{S}_n,\mathcal{S}_n}(\lambda)
  \end{pmatrix}.
\end{align}
The blocks $\mathbf{P}^\varepsilon_{\mathcal{S}_i,\mathcal{S}_j}$ of $\mathbf{P}^\e$ determine the transitions of the Markov chain
$Z^\e$
between the wells $\mathcal{S}_i$ and $\mathcal{S}_j$. Only the blocks $\mathbf{A}^\varepsilon_{\mathcal{S}_i,\mathcal{S}_i}(\lambda)$, 
$i=1,\dots,n$, 
that include the main diagonal depend on $\lambda$.

For the further analysis of the matrix $\mathbf{A}^\varepsilon(\lambda)$ it is helpful to study the limit behaviour of its entries
$a_{x,y}^\e(\lambda)$ as $\e\to 0$.

The following expansions follow immediately from Lemma \ref{l:proptransitionprobabilities}. 
\begin{lem}\label{ordnungeintraege} 
 Let $\lambda\in\mathbb C$, then
\begin{itemize}
 \item[(i)] 	for $x,y\in\mathcal{S}$ and $y\notin\left\{x,y^*(x)\right\}$ we have $a^\varepsilon_{x,y}=d_{x,y}(1+\mathcal O(\e^\alpha))$;
 \item[(ii)] 	for $x\in\mathcal{S}\setminus\mathcal{M}$ and $y=y^*(x)$ 
 we have $a^\varepsilon_{x,y^*(x)}(\lambda)=\mathcal O(\varepsilon^{-\alpha})$, 
more precisely, 
\begin{align}
a^\varepsilon_{x,y^*(x)}(\lambda)=\frac{1}{\varepsilon^\alpha h}-\tilde{a}^\varepsilon_{x,y^*(x)}
\quad \text{where}\quad \tilde{a}^\varepsilon_{x,y^*(x)}=f_{x,y^*(x)}(1+\mathcal O(\e^\alpha));
\end{align}
 \item[(iii)]  	for $x\in\mathcal{S}\setminus\mathcal{M}$ we have  $a^\varepsilon_{x,x}(\lambda)=\mathcal O(\varepsilon^{-\alpha}$),
more precisely,  
\begin{align}
a^\varepsilon_{x,x}(\lambda)=-\frac{1}{\varepsilon^\alpha h}-\lambda+\tilde{a}^\varepsilon_{x,x}
\quad \text{where}\quad \tilde{a}^\varepsilon_{x,x}=d_{x,x}(1+\mathcal O(\e^\alpha));
\end{align}
 \item[(iv)]	
 for $x\in\mathcal{M}$ we have $a^\varepsilon_{x,x}=\mathcal O(1)$, more precisely, 
 \begin{align} a^\varepsilon_{x,x}=-\tilde{a}^\varepsilon_{x,x}-\lambda
 \quad\text{where}\quad\tilde{a}^\varepsilon_{x,x}=f_{x,x}(1+\mathcal O(\e^\alpha)).
 \end{align}
\end{itemize}
As in Lemma \ref{l:proptransitionprobabilities}, 
the bounds hidden in the Landau symbols $\mathcal O$ are uniform over $x,y\in\mathcal S$.
\end{lem}

To get a better understanding of the structure of $\mathbf A^\e(\lambda)$,  let us  for example take a 
detailed look at the block $ \mathbf{A}^\varepsilon_{\mathcal{S}_1,\mathcal{S}_1}(\lambda)$. 
If we number the states of $\mathcal S_1$ in the increasing order by $x_1,\ldots, x_{N(1)}$,  
then this block has the form
\begin{align}
 \begin{pmatrix}
  -\frac{1}{\varepsilon^\alpha h}-\lambda+\tilde{a}^\varepsilon_{x_1,x_1}	& \cdots 	& \frac{1}{\varepsilon^\alpha h}-\tilde{a}^\varepsilon_{x_1,y^*(x_1)}  	&\cdots 	& a^\varepsilon_{x_1,\mathfrak{m}_1} 				& \cdots 	&a^\varepsilon_{x_1,x_{N(1)}}\\
  \vdots 							     	&       	& \vdots        							&  		& \cdots      					&           	&\vdots\\
 a^\varepsilon_{\mathfrak{m}_1,x_1}							&		& a^\varepsilon_{\mathfrak{m}_1,y^*(x_1)}						&		&-\tilde{a}^\varepsilon_{\mathfrak{m}_1,\mathfrak{m}_1}-\lambda	&		&a^\varepsilon_{\mathfrak{m}_1,x_{N(1)}}\\
\vdots									&		& \vdots								&								& \cdots					&		&\vdots\\
a^\varepsilon_{x_{N(1)},x_1}  							&		&a^\varepsilon_{x_{N(1)},y^*(x_1)}						&		&a^\varepsilon_{x_{N(1)},\mathfrak{m}_1}				&\cdots		&-\frac{1}{\varepsilon^\alpha h}-\lambda+\tilde{a}^\varepsilon_{x_{N(1)},x_{N(1)}}\\      
 \end{pmatrix}
\end{align}
Note that the numbers $\tilde{a}^\varepsilon_{x,y}$ that appear in the previous lemma are all positive.

Let us establish the connection between the matrix 
$\frac{1}{\varepsilon^\alpha h}(\mathbf{P}^\varepsilon-\mathbf{I_N})$ and the matrix $\mathbf{Q}$ of the limit Markov chain appearing in the
Theorems \ref{th:1} and \ref{th:2}.

\begin{lem}\label{approxqij}
Let $\mathbf{Q}=(q_{i,j})_{i,j=1}^n$ be the generator of the Markov chain $Y$ on the state space $\mathcal M$ defined in
\eqref{eintraegeQ}. 
Then there exist constants $\hat{C}=\hat{C}(\delta, h, R,\alpha)$ and 
 $\varepsilon_0>0$ such that for every $0<\varepsilon<\varepsilon_0$, 
every
$1\leq i,j\leq n$, $i\neq j$, 
 \begin{align}
  \Big| q_{i,j}-\frac{1}{\varepsilon^\alpha h}\sum_{y\in\mathcal{S}_j}p^\varepsilon_{\mathfrak{m}_i,y}\Big|
  \leq \hat{C}\varepsilon^\alpha
 \end{align}
 and 
\begin{align}
  \Big| q_{i,i}-\frac{1}{\varepsilon^\alpha h}\sum_{y\in\mathcal{S}_j}(p^\varepsilon_{\mathfrak{m}_i,y}-\delta_{\mathfrak{m_i}}^y)\Big|
  \leq \hat{C}\varepsilon^\alpha.
\end{align}
\end{lem}

\begin{proof}
Let for definiteness $1\leq i<j<n$. Then according to \eqref{deftransprobs1}, \eqref{eq:exp} and \eqref{eintraegeQ}
\begin{equation}
\begin{aligned}
\frac{1}{\varepsilon^\alpha h}\sum_{y\in\mathcal{S}_j}p^\varepsilon_{\mathfrak{m}_i,y}
&=\frac{1}{\varepsilon^\alpha h}
\sum_{y\in\mathcal{S}_j}\mathbb P\Big(\mathfrak{m}_i -hU'(\mathfrak{m}_i)+\e^{\frac{1}{\alpha}} h L_1 \in I_y\Big)\\
&=\frac{1}{\varepsilon^\alpha h}
\mathbb P\Big(\mathfrak{m}_i +\e^{\frac{1}{\alpha}} h L_1 \in [\mathfrak{s}_{j-1} , \mathfrak{s}_j )   \Big)\\
&=\frac{1}{2}\Big(\frac{1}{(\mathfrak{s}_{j-1}-\mathfrak{m}_{i})^\alpha}-\frac{1}{(\mathfrak{s}_{j}-\mathfrak{m}_{i})^\alpha}\Big)
+\mathcal O(\e^\alpha)
=
 q_{i,j}+\mathcal O(\e^\alpha).
\end{aligned}
\end{equation}
Analogously, for $j=n$, with the help of \eqref{deftransprobs3}, \eqref{eq:exp} and \eqref{eintraegeQ} we get
\begin{equation}
\begin{aligned}
\frac{1}{\varepsilon^\alpha h}\sum_{y\in\mathcal{S}_n}p^\varepsilon_{\mathfrak{m}_i,y}
&=\frac{1}{\varepsilon^\alpha h}
\sum_{y\in\mathcal{S}_n}\mathbb P\Big(\mathfrak{m}_i -hU'(\mathfrak{m}_i)+\e^{\frac{1}{\alpha}} h L_1 \in I_y\Big)\\
&=\frac{1}{\varepsilon^\alpha h}
\mathbb P\Big(\mathfrak{m}_i +\e^{\frac{1}{\alpha}} h L_1 \in [\mathfrak{s}_{n-1} , +\infty)   \Big)\\
&=\frac{1}{2}\frac{1}{(\mathfrak{s}_{j-1}-\mathfrak{m}_{i})^\alpha} +\mathcal O(\e^\alpha)
= q_{i,n}+\mathcal O(\e^\alpha),
\end{aligned}
\end{equation}
and for $i=j$
\begin{equation}
\begin{aligned} 
\frac{1}{\varepsilon^\alpha h} \sum_{y\in\mathcal{S}_i}\Big( p^\varepsilon_{\mathfrak{m}_i,y}-\delta_{\mathfrak{m_i}}^y\Big)
&=\frac{1}{\varepsilon^\alpha h}
\Big(-1+\sum_{y\in\mathcal{S}_i}\mathbb P\Big(\mathfrak{m}_i -hU'(\mathfrak{m}_i)+\e^{\frac{1}{\alpha}} h L_1 \in I_y\Big)\Big)\\
&=-\frac{1}{\varepsilon^\alpha h} 
\mathbb P\Big( \mathfrak{m}_i +\e^{\frac{1}{\alpha}} h L_1 \notin [\mathfrak{s}_{i-1} , \mathfrak{s}_i )  \Big)\\
&=-\frac{1}{2}\Big(\frac{1}{(\mathfrak{m}_{i}-\mathfrak{s}_{i-1})^\alpha}+\frac{1}{(\mathfrak{s}_{j}-\mathfrak{m}_{i})^\alpha}\Big)
+\mathcal O(\e^\alpha)
=
 q_{i,i}+\mathcal O(\e^\alpha).
\end{aligned} 
\end{equation}
\end{proof}

Now let us come to the investigation of the eigenvalues. We start with a well known result.

\begin{prp}\label{locationeigenvalues}
 All eigenvalues of $\mathbf{Q}$ and $\mathbf{Q}^\varepsilon$ have non-positive real parts.
\end{prp}
\begin{proof} This result follows from the Gershgorin circle theorem 
(see, for example, Theorem 1.11 in \cite{varga2009}). Indeed, since  $q_{i,i}<0$ for all $1\leq i\leq n$,
and $q^\varepsilon_{x,x}<0$ for all $x\in\mathcal{S}$, as well as
\begin{align}
 \sum_{j\neq i}q_{i,j}=-q_{i,i}\quad\text{and}\quad \sum_{y\neq x}q^\varepsilon_{x,y}=-q^\varepsilon_{x,x}
\end{align}
all the Gershgorin circles are located in the negative complex half plane 
and touch the origin. In particular the principal eigenvalues equal to zero, $\lambda^\varepsilon_1=\lambda^{\mathbf{Q}}_1=0$.
\end{proof}

The proof of Theorem \ref{theoremeigenwerte} mainly consists in the investigation of the asymptotics of the characteristic polynomial
$P^\varepsilon(\lambda)\coloneqq \det(\mathbf{A}^\varepsilon(\lambda))$ in the limit $\e\to 0$. 
We will show that, after an appropriate scaling, it converges uniformly on a sufficiently
large disc to the characteristic polynomial
$P^{\mathbf{Q}}(\lambda)\coloneqq \det(\mathbf{A}^{\mathbf{Q}}(\lambda))$ 
where $\mathbf{A}^{\mathbf{Q}}(\lambda)\coloneqq \mathbf{Q}-\lambda\mathbf{I}_n$.

\begin{prp}\label{proppolynome} 
There exist constants $\varepsilon_0>0$, and $C=C(h,\delta,\alpha,R)>1$
such that for every 
$0<\varepsilon<\varepsilon_0$
and $\lambda\in\mathbb{C}$
\begin{align}
 \Big| (-\varepsilon^\alpha h)^{N-n}P^\varepsilon(\lambda)-P^\mathbf{Q}(\lambda)\Big|
 \leq C(|\lambda|+ C)^{n+1}\varepsilon^\alpha.
\end{align}
\end{prp}

Before we start with the proof of this Proposition let us show how it implies the statements of Theorem \ref{theoremeigenwerte}.

\noindent
\textit{Proof of Theorem \ref{theoremeigenwerte}.} Both statements follow from Proposition \ref{proppolynome} and the 
Hurwitz Theorem (see, for example,
Theorem 1.3.8 in \cite{RahmSchm2002}):
\begin{thm}[Hurwitz]
 Let $G \subset \mathbb C$ be a region and let $\{f_k\}_{k\geq 1}$ be a sequence of analytic functions on $G$ that
converges to a non-zero function $f$ uniformly on every compact subset of $G$. Then $z_0 \in G$
is a zero of $f$ with multiplicity $m$ if and only if there exists a neighbourhood $H \subset G$ such that,
in every ball $\{z\in\mathbb C\colon |z-z_0|\leq R \} \subset H$, each function $f_k$ whose index exceeds some bound $k_0 = k_0 (R)$
has exactly $m$ zeros, counted according to their multiplicities.
\end{thm}

For the proof of (i) fix a closed ball centred around $0$ 
that includes all the eigenvalues of $\mathbf{Q}$. Then Proposition \ref{proppolynome} ensures the uniform convergence inside of this ball. 
Moreover, let a sequence $r_\varepsilon>0$ be such that
\begin{align}\label{spectralgap}
 \lim_{\varepsilon\rightarrow 0} r_\varepsilon=+\infty\quad\text{and}\quad 
 \lim_{\varepsilon\rightarrow 0}\varepsilon^\alpha r_\varepsilon^{n+1} = 0.
\end{align}
Then, Proposition \ref{proppolynome} implies that
\begin{align}
 \lim_{\varepsilon\rightarrow 0}\sup_{|\lambda|\leq r_\varepsilon}
 |(- \varepsilon^\alpha h)^{N-n}P^\varepsilon(\lambda)-P^\mathbf{Q}(\lambda)|=0
\end{align}
and that proves the statement (ii).
\hfill (\textit{Theorem \ref{theoremeigenwerte}})$\Box$

\medskip

The rest of this section  is devoted to the proof of Proposition \ref{proppolynome}.
We start of with some preparations. Recall that the set $\mathcal S$ consists of $N$ elements, 
let 
\begin{align}
S_N=
S_{|\mathcal{S}|}\coloneqq 
\{\boldsymbol{\pi}=(\pi_x)_{x\in\mathcal{S}}\colon \boldsymbol{\pi} \text{ is bijective mapping } \mathcal{S}\rightarrow \mathcal{S}\}
\end{align}
be the set of all permutations of $\mathcal{S}$, and denote by $\textbf{id}$ the identity permutation.
The Leibniz formula for the characteristic polynomial $P^\varepsilon(\lambda)$ gives
\begin{align}\label{Leibniz}
 P^\varepsilon(\lambda)=\sum_{\boldsymbol{\pi}\in S_{N}}\mathop{\mathrm{sgn}}(\boldsymbol{\pi})\prod_{x\in\mathcal{S}} 
 a^\varepsilon_{x,\pi_x}(\lambda).
\end{align}

Let us briefly explain the heuristics behind the following analysis of $P^\varepsilon$. 
Let the parameters $h$, $\delta$, $R$ and $\alpha$ be fixed. 
Clearly, in view of the asymptotics of the entries $a_{x,y}^\e$ as $\e\to 0$, for any fixed $\lambda\in\mathbb{C}$ we have 
$| P^\varepsilon(\lambda)|\to \infty$ as
$\varepsilon\rightarrow 0$. 
With that in mind, let us omit for a moment the values $\tilde{a}^\varepsilon_{x,y}$ since they are of the order $\mathcal O(1)$,
that is negligible in comparison to $\varepsilon^{-\alpha}$.
Then one can think of $P^\varepsilon(\lambda)$ as of a polynomial of the variable $\varepsilon^{-\alpha}$
with coefficients depending on $\lambda$ and terms of order $\mathcal O(1)$. 

The degree of this polynomial does not exceed $N-n$ since only the rows indexed by the minima 
$\mathfrak{m}_1,\ldots,\mathfrak{m}_n$ do not contain the term $(\varepsilon^{\alpha}h)^{-1}$. Hence we can write
\begin{align}
 P^\varepsilon(\lambda)\approx \sum_{j=0}^{N-n} c^\varepsilon_j(\lambda)(\varepsilon^{\alpha}h)^{-j}.
\end{align}
By a multiplication by the factor $(\varepsilon^{\alpha}h)^{N-n}$ we single out the coefficient 
$c^\varepsilon_{N-n}(\lambda)$ since all the other summands vanish in
the limit $\varepsilon\rightarrow 0$.
On the other hand, $c^\varepsilon_{N-n}(\lambda)$ contains expressions of the form 
$\sum_{y\in\mathcal{S}_j}\frac{1}{\varepsilon^\alpha h}p^\varepsilon_{\mathfrak{m}_i,y}$, $1\leq i,j\leq n$,
which, as we saw in Lemma \ref{approxqij}, are approximations of the entries $q_{i,j}$ of the generator matrix $\mathbf{Q}$. 
Finally, having that in mind we prove
that $c^\varepsilon_{N-n}(\lambda)$ approximates $P^{\mathbf{Q}}(\lambda)$.

In view of the Leibniz formula \eqref{Leibniz} we should first turn to the permutations that contribute the 
terms of order $\mathcal O((\varepsilon^{\alpha}h)^{-(N-n)})$
to the polynomial $P^\varepsilon$. 
As we see from Lemma \ref{ordnungeintraege} (ii) and (iii), such permutations are precisely those which satisfy the condition
$\pi_x=x$ or $\pi_x=y^*(x)$ for $x\in\mathcal{S}\setminus\mathcal{M}$. Consider the family of permutations 
\begin{align}\label{Pi0}
 \Pi_0\coloneqq \left\{\boldsymbol{\pi}=(\pi_x)_{x\in\mathcal{S}}\in S_{N}\colon \pi_x=x\text{ or }\pi_x=y^*(x)\text{ for all }x\in\mathcal{S}\setminus\mathcal{M}\right\}.
\end{align}

\begin{lem}
\label{lemmapermu1}
For every $i=1,\dots, n$, every $y\in\mathcal{S}_i$ there exists a permutation $\boldsymbol{\pi}\in\Pi_0$ such that
 \begin{align}
  \pi_{\mathfrak{m}_i}=y.
 \end{align}
\end{lem}
\begin{proof}
Fix $y\in\mathcal{S}_i$. If $y=\mathfrak{m}_i$, then we can choose $\boldsymbol{\pi}=\textbf{id}$. Otherwise there exists a $k=k(y)\geq 1$
such that $((\mathbf{P}^0)^{(k)})_{y,\mathfrak{m}_i}=1$, that is the deterministic motion, if
started in $y$, reaches the minimum $\mathfrak{m}_i$ in $k$ steps. Recall that the mapping 
$T\colon \mathcal{S}\rightarrow \mathcal{S}$,  $T(x)\coloneqq y^*(x)$. 
Then the equality $((\mathbf{P}^0)^{(k)})_{y,\mathfrak{m}_i}=1$ is
equivalent to
\begin{align}
 T^ky
 =\mathfrak{m}_i.
\end{align}
Now we can define a permutation $\hat{\boldsymbol{\pi}}\colon\mathcal{S}\rightarrow \mathcal{S}$ by 
\begin{align}
\hat{\pi}_y\coloneqq Ty ,\ \hat{\pi}_{Ty}=T^2y,\dots,\ \hat{\pi}_{T^{k-1}y}=T^ky=\mathfrak{m}_i,\ \hat{\pi}_{\mathfrak{m}_i}=y
\end{align}
as well as
\begin{align}
 \hat{\pi}_z=z\quad\text{if } z\in\mathcal{S}\setminus\{y,Ty,\ldots,T^ky\},
\end{align}
It is clear that $\hat{\boldsymbol{\pi}}\in\Pi_0$.
\end{proof}

Consider the family of  vectors
\begin{align}
 \mathfrak{T}_n\coloneqq \left\{(y_1,\ldots,y_n)\colon y_i\in\mathcal{S}_i\right\}.
\end{align}

\begin{lem}\label{lemmapermu2}
 For every $(y_1,\ldots,y_n)\in\mathfrak{T}_n$ there exists a permutation $\boldsymbol{\pi}\in\Pi_0$ such that
 \begin{align}
  \pi_{\mathfrak{m}_i}=y_i,\quad1\leq i\leq n.
 \end{align}
\end{lem}

\begin{proof}
Let $\boldsymbol{\pi}_i$, $i=1,\dots,n$, be the permutations constructed in the proof of Lemma \ref{lemmapermu1}. Note that
each of them is non-identical only on the set $\mathcal{S}_i$. Hence a composition of $\boldsymbol{\pi}_i$, $i=1,\dots,n$,
has the desired property.
\end{proof}

\begin{lem}\label{lemmapermu3}
 For every $(y_1,\ldots,y_n)\in\mathfrak{T}_n$ there exists a permutation $\boldsymbol{\pi}\in\Pi_0$ such that
 \begin{align}
  \pi_{\mathfrak{m}_i}=y_j,\quad 1\leq i,j\leq n.
 \end{align}
\end{lem}
\begin{proof} 
Fix $(y_1,\ldots,y_n)\in\mathfrak{T}_n$ and an associated permutation $\boldsymbol{\pi}$ from Lemma \ref{lemmapermu2} 
such that $\pi_{\mathfrak{m}_i}=y_i$, $i=1,\dots,n$.
Since no condition on the images of the states $\mathfrak{m}_1,\dots,\mathfrak{m}_n$
is imposed in the definition of $\Pi_0$, we can permute them among each other to obtain $\boldsymbol{\pi}\in\Pi_0$ with the 
required property.
\end{proof}

For the investigation of the determinant of the matrix $\mathbf{A}^\e(\lambda)$
we need to introduce the following disjoint decomposition of the set of all permutations $S_N$. 
For $p=0,\dots,N-n$ let
\begin{equation}
\label{Pip}
\begin{aligned}
\Pi_p \coloneqq \{\boldsymbol{\pi}\in S_N\colon &\text{ there are precisely }p\text{ states }
x_1,\ldots,x_p\in\mathcal{S}\setminus\mathcal{M} \\
&\text{ such that } \pi_{x_q}\notin\left\{x_q,y^*(x_q)\right\}\text{ for all }1\leq q\leq p\}.
\end{aligned}
\end{equation}
With this definition we can write
\begin{align}
 S_N=\bigsqcup_{p=0}^{N-n}\Pi_p.
\end{align}
Also note that the definition of $\Pi_0$ in \eqref{Pi0} coincides with the one given in \eqref{Pip}.
For any $\boldsymbol{\pi}\in S_N$ we also define
\begin{equation}
\begin{aligned}
 \mathcal{S}^1(\boldsymbol{\pi})	&\coloneqq 	\{x\in\mathcal{S}\setminus\mathcal{M}\colon \pi_x=x\},\\
 \mathcal{S}^2(\boldsymbol{\pi})	&\coloneqq 	\{x\in\mathcal{S}\setminus\mathcal{M}\colon \pi_x=y^*(x)\},\\
 \mathcal{S}^3(\boldsymbol{\pi})	&\coloneqq 	
 \mathcal{S}\setminus\Big(\mathcal{S}^1(\boldsymbol{\pi})\bigsqcup\mathcal{S}^2(\boldsymbol{\pi})\Big).
\end{aligned}
\end{equation}
Clearly these sets are disjoint and $\mathcal S=\bigsqcup_{i=1}^3 \mathcal{S}^i(\boldsymbol{\pi})$.

The following Lemma follows directly from the definition of $\Pi_p$.
\begin{lem}\label{formelpermus}
\begin{itemize}
 \item [(i)]  For every $\boldsymbol{\pi}\in\Pi_0$ we have $\mathcal{S}^3(\boldsymbol{\pi})=\mathcal{M}$.
 \item [(ii)] For every $0\leq p\leq N-n$ and every $\boldsymbol{\pi}\in\Pi_p$ 
 we have $\vert\mathcal{S}^1(\boldsymbol{\pi})\sqcup\mathcal{S}^2(\boldsymbol{\pi})\vert=N-n-p$.\hfill \qed
\end{itemize}
\end{lem}

\medskip

Let us introduce functions $f_p\colon \boldsymbol{\pi} \to \mathbb N_0$ that count the number of fixed points of 
$\boldsymbol{\pi}\in\Pi_p$ on the set $\mathcal{S}\setminus\mathcal{M}$, namely
\begin{align}
 f_p(\boldsymbol{\pi})\coloneqq |\{x\in\mathcal{S}\setminus\mathcal{M}\colon \pi_x=x\}|
 =|\mathcal{S}^1(\boldsymbol{\pi})|.
\end{align}
By definition of $\Pi_p$, it follows that $f_p$ has values in $\left\{0,\ldots,N-n-p\right\}$.
We obtain a trivial equality
\begin{align}
\prod_{x\in\mathcal{S}^1(\boldsymbol{\pi})}\Big(-\frac{1}{\varepsilon^\alpha h}-\lambda+\tilde{a}^\varepsilon_{x,x}\Big)
=(-1)^{f_p(\boldsymbol{\pi})}
\prod_{x\in\mathcal{S}^1(\boldsymbol{\pi})}\Big(\frac{1}{\varepsilon^\alpha h}+(\lambda-\tilde{a}^\varepsilon_{x,x})\Big)
\end{align}
which allows us to write the characteristic polynomial $P^\varepsilon$ as
\begin{equation}
\label{charpoly1}
\begin{aligned}
P^\varepsilon(\lambda)
=\sum_{p=0}^{N-n}\Big[\sum_{\boldsymbol{\pi}\in\Pi_p}\mathop{\mathrm{sgn}}(\boldsymbol{\pi})(-1)^{f_p(\boldsymbol{\pi})}
&\cdot\prod_{x\in\mathcal{S}^1(\boldsymbol{\pi})}\Big(\frac{1}{\varepsilon^\alpha h}+(\lambda-\tilde{a}^\varepsilon_{x,x})\Big)\\
&\times\prod_{x\in\mathcal{S}^2(\boldsymbol{\pi})}\Big(\frac{1}{\varepsilon^\alpha h}-\tilde{a}^\varepsilon_{x,y^*(x)}\Big)\\
&\times\prod_{x\in\mathcal{S}^3(\boldsymbol{\pi})}a^\varepsilon_{x,\pi_x}\Big].
\end{aligned}
\end{equation}
Using Lemma \ref{formelpermus} (ii) along with a simple straightforward calculation 
yields that there are numbers $\beta^\varepsilon_{l,p,\boldsymbol{\pi}}(\lambda)=\mathcal O(1)$ such that
\begin{align}\label{betas}
 \prod_{x\in\mathcal{S}^1(\boldsymbol{\pi})}\Big(\frac{1}{\varepsilon^\alpha h}+(\lambda-\tilde{a}^\varepsilon_{x,x})\Big)
 \prod_{x\in\mathcal{S}^2(\boldsymbol{\pi})}\Big(\frac{1}{\varepsilon^\alpha h}-\tilde{a}^\varepsilon_{x,x}\Big)
 =\sum_{l=0}^{N-n-p}\beta^\varepsilon_{l,p,\boldsymbol{\pi}}(\lambda)\Big(\frac{1}{\varepsilon^{\alpha}h}\Big)^l.
\end{align}
As explained before $P^\varepsilon$ can roughly be viewed as a polynomial in the variable $(\varepsilon^{\alpha}h)^{-1}$ and our aim is 
to single out the coefficient that belongs to the highest power.

After multiplying the polynomial $P^\e$ by the factor
$(\varepsilon^{\alpha} h)^{N-n}$ and taking Lemma \ref{formelpermus} (i) into account, 
we split the sum in \eqref{charpoly1} up into three parts to get the representation
\begin{equation}
\label{charpoly2}
\begin{aligned}
(\varepsilon^{\alpha} h)^{N-n}P^\varepsilon(\lambda)
&=\sum_{\boldsymbol{\pi}\in\Pi_0}\mathop{\mathrm{sgn}}(\boldsymbol{\pi})(-1)^{f_0(\boldsymbol{\pi})}
\prod_{i=1}^na^\varepsilon_{\mathfrak{m}_i,\pi_{\mathfrak{m}_i}}(\lambda)\\
&\hspace*{-1.5cm}+ \sum_{\boldsymbol{\pi}\in\Pi_0}
\mathop{\mathrm{sgn}}(\boldsymbol{\pi})(-1)^{f_0(\boldsymbol{\pi})}
\Big(\sum_{l=0}^{N-n-1}\beta^\varepsilon_{l,0,\boldsymbol{\pi}}(\lambda)(\varepsilon^{\alpha}h)^{N-n-l}\Big)
\prod_{i=1}^na^\varepsilon_{\mathfrak{m}_i,\pi_{\mathfrak{m}_i}}(\lambda)\\
&\hspace*{-1.5cm}+ \sum_{p=1}^{N-n}\sum_{\boldsymbol{\pi}\in\Pi_p}\mathop{\mathrm{sgn}}(\boldsymbol{\pi})(-1)^{f_p(\boldsymbol{\pi})}
\Big(\sum_{l=0}^{N-n-p}\beta^\varepsilon_{l,p,\boldsymbol{\pi}}(\lambda)(\varepsilon^{\alpha}h)^{N-n-l}\Big)
\prod_{x\in\mathcal{S}^3(\boldsymbol{\pi})}a^\varepsilon_{x,\pi_x}.
\end{aligned}
\end{equation}
Let us continue the investigation of the permutations.

\begin{lem}\label{lemmapermu4}
 Fix $\boldsymbol{\pi}\in\Pi_0$ such that $\pi_{\mathfrak{m}_i}\in\mathcal{S}_i$ for every $1\leq i\leq n$. Then
 \begin{align}
  \mathop{\mathrm{sgn}}(\boldsymbol{\pi})(-1)^{f_0(\boldsymbol{\pi})}=(-1)^{N-n}.
 \end{align}
\end{lem}
\begin{proof} Fix $1\leq i\leq n$ and let $\pi_{\mathfrak{m}_i}=y_i\in\mathcal{S}_i$. 
From this condition and from the definition of $\Pi_0$ it follows that for every $x\in\mathcal{S}_i$
we have $\pi_x\in\mathcal{S}_i$. Hence, we can write $\boldsymbol{\pi}$ as a product of $n$
disjoint permutation cycles
\begin{align}
 \boldsymbol{\pi}=\xi_1\cdot\ldots\cdot\xi_n,
\end{align}
where each cycle $\xi_i$ acts only on $\mathcal{S}_i$. Define $k_i=k(y_i)\geq 0$ as the number of steps the deterministic motion needs to reach $\mathfrak{m}_i$ if started in $y_i$
(see the proof of Lemma \ref{lemmapermu1}). Then we can write each cycle $\xi_i$ as
\begin{align}
 \xi_i=\begin{cases}
        y_i\mapsto Ty_i\mapsto \cdots\mapsto T^{k_i-1}y_i\mapsto \mathfrak{m}_i\mapsto y_i,& k_i\geq 1,\\
        \mathfrak{m}_i\mapsto \mathfrak{m}_i,& k_i=0.
       \end{cases}
\end{align}
If case $k_i=0$ then $\boldsymbol{\pi}$ acts as identity in $\mathcal{S}_i$, i.e.\ $\pi_x=x$
for all $x\in\mathcal{S}_i$. For the consistency of notation in this case, however, we do not omit the identity cycle $\xi_i$ having in mind that
the sign of an identity cycle is $1$.
Let $\vert\xi_i\vert$ denote the length of the $i$-th cycle. Then  $\vert\xi_i\vert=k_i+1$ and
\begin{align}
 \mathop{\mathrm{sgn}}(\boldsymbol{\pi})=(-1)^{\vert\xi_1\vert-1}\cdot\ldots\cdot(-1)^{\vert\xi_n\vert-1}=(-1)^{\sum_{i=1}^n k_i}.
\end{align}
On the other hand, from the definition of $f_0$ it follows that
\begin{align}
 f_0(\boldsymbol{\pi})=\sum_{i=1}^n\left(\vert\mathcal{S}_i\vert-(k_i+1)\right)=N-n-\sum_{i=1}^n k_i
\end{align}
which proves the claim.
\end{proof}

We construct a further decomposition of the set $\Pi_0$. 
Let $S_n$ denote the set of all permutations 
$\boldsymbol{\sigma}\colon \{1,\dots,n\}\rightarrow\{1,\dots,n\}$
and set
\begin{align}
 \Pi(\boldsymbol{\sigma})\coloneqq \{\boldsymbol{\pi}\in\Pi_0\colon \pi_{\mathfrak{m}_i}\in\mathcal{S}_{\sigma_i},\ 1\leq i\leq n\}.
\end{align}
Then it is easy to see that
\begin{align}
 \Pi_0=\bigsqcup_{\boldsymbol{\sigma}\in S_n}\Pi(\boldsymbol{\sigma}).
\end{align}

\begin{lem}
\label{lemmapermu5}
Let $\boldsymbol{\sigma}\in S_n$. Then, for every $\boldsymbol{\pi}\in\Pi(\boldsymbol{\sigma})$
 \begin{align}
 \mathop{\mathrm{sgn}}(\boldsymbol{\sigma})(-1)^{N-n}=\mathop{\mathrm{sgn}}(\boldsymbol{\pi})(-1)^{f_0(\boldsymbol{\pi})} .
 \end{align}
\end{lem}
\begin{proof} 
Let $\boldsymbol{\sigma}=\zeta_1\cdot\ldots\cdot\zeta_L$ be a representation of 
$\boldsymbol{\sigma}$ as a product of  $1\leq L\leq n$ disjoint cycles.   
In particular, $|\zeta_1|+\cdots+|\zeta_L|=n$.

From the definition of $\Pi(\boldsymbol{\sigma})$ it follows that there exists a collection of states $(y_1,\ldots,y_n)\in\mathfrak{T}_n$ such that
\begin{align}
 \pi_{\mathfrak{m}_1}=y_{\sigma_1}\in\mathcal{S}_{\sigma_1},\ldots,\pi_{\mathfrak{m}_n}=y_{\sigma_n}\in\mathcal{S}_{\sigma_n}.
\end{align}
Now, every $\boldsymbol{\pi}\in\Pi_0$ can be written as a composition $\boldsymbol{\pi}=\boldsymbol{\pi}_\sigma\circ\tilde{\boldsymbol{\pi}}$
where $\tilde{\boldsymbol{\pi}}\in\Pi_0$ such that
$\tilde{\pi}_{\mathfrak{m}_i}\in\mathcal{S}_i$ and where $\boldsymbol{\pi}_\sigma$ only acts on the states
$y_1=\tilde{\pi}_{\mathfrak{m}_1},\dots, y_n=\tilde{\pi}_{\mathfrak{m}_n}$ according to the permutation $\boldsymbol{\sigma}$.
So let $\sigma(y_1,\ldots,y_n)$ denote a cycle representation of $\boldsymbol{\pi}_\sigma$.
With the notations introduced in the previous lemmas we then can write $\boldsymbol{\pi}$ as
\begin{align}
 \boldsymbol{\pi}=\sigma(y_1,\ldots,y_n)
\Big(\mathfrak{m}_1\mapsto y_1\mapsto\cdots\mapsto T^{k_1-1}y_1\Big)\cdot\ldots\cdot
\Big(\mathfrak{m}_n\mapsto y_n\mapsto \cdots\mapsto T^{k_n-1}y_n\Big).
\end{align}
Keep in mind that the numbers $k_i$ also depend on $y_i$ and that one always has $T^{k_i}y_i=\mathfrak{m}_i$. 
Since $f_0$ does not take $\mathfrak{m}_1,\ldots,\mathfrak{m}_n$ into account
we have $f_0(\boldsymbol{\pi})=f_0(\tilde{\boldsymbol{\pi}})$.
Combining that with the equality $\mathop{\mathrm{sgn}}(\boldsymbol{\sigma})=\mathop{\mathrm{sgn}}(\boldsymbol{\pi}_\sigma)$ 
and Lemma \ref{lemmapermu5} yields
\begin{align}
 \mathop{\mathrm{sgn}}(\boldsymbol{\pi})(-1)^{f_0(\boldsymbol{\pi})}=\mathop{\mathrm{sgn}}(\boldsymbol{\pi}_\sigma)\mathop{\mathrm{sgn}}(\tilde{\boldsymbol{\pi}})(-1)^{f_0(\tilde{\boldsymbol{\pi}})}=\mathop{\mathrm{sgn}}(\boldsymbol{\sigma})(-1)^{N-n}.
\end{align}
\end{proof}

Lemma \ref{lemmapermu4} allows us to rewrite the first summand in \eqref{charpoly2} as
\begin{equation}
\begin{aligned}
\label{charpoly5}
 \sum_{\boldsymbol{\pi}\in\Pi_0}\mathop{\mathrm{sgn}}(\boldsymbol{\pi})(-1)^{f_0(\boldsymbol{\pi})}
 \prod_{i=1}^na^\varepsilon_{\mathfrak{m}_i,\pi_{\mathfrak{m}_i}}(\lambda) 
 &=\sum_{\boldsymbol{\sigma}\in S_n}\sum_{\boldsymbol{\pi}\in\Pi(\boldsymbol{\sigma})} 
\mathop{\mathrm{sgn}}(\boldsymbol{\pi})(-1)^{f_0(\boldsymbol{\pi})}\prod_{i=1}^na^\varepsilon_{\mathfrak{m}_i,\pi_{\mathfrak{m}_i}}(\lambda)\\
&=   \sum_{\boldsymbol{\sigma}\in S_n}\mathop{\mathrm{sgn}}(\boldsymbol{\sigma})(-1)^{N-n}
\sum_{\boldsymbol{\pi}\in\Pi(\boldsymbol{\sigma})}\prod_{i=1}^na^\varepsilon_{\mathfrak{m}_i,\pi_{\mathfrak{m}_i}}(\lambda)\\
&=   (-1)^{N-n}\sum_{\boldsymbol{\sigma}\in S_n}\mathop{\mathrm{sgn}}(\boldsymbol{\sigma})
\sum_{y_1\in\mathcal{S}_{\sigma_1}}\cdots\sum_{y_n\in\mathcal{S}_{\sigma_n}}
\prod_{i=1}^na^\varepsilon_{\mathfrak{m}_i,y_i}(\lambda)\\
&=  (-1)^{N-n}\sum_{\boldsymbol{\sigma}\in S_n}\mathop{\mathrm{sgn}}(\boldsymbol{\sigma})
\prod_{i=1}^n\sum_{y\in\mathcal{S}_{\sigma_i}}a^\varepsilon_{\mathfrak{m}_i,y}(\lambda).
\end{aligned}
\end{equation}
This representation makes clear why we also included the term $(-1)^{N-n}$ into the normalization factor 
for $P^\varepsilon$ in Proposition \ref{proppolynome}.

The next lemma provides us with the crucial estimate.

\begin{lem}
\label{abschpoly}
There exist constants $\varepsilon_0>0$ and $\tilde{C}=\tilde{C}(h,\delta,R,\alpha,n)>1$ such that 
for every $\lambda\in\mathbb{C}$ and every $0<\varepsilon<\varepsilon_0$ 
\begin{align}\label{charpoly3}
 \Big|\sum_{\boldsymbol{\sigma}\in S_n}\mathop{\mathrm{sgn}}(\boldsymbol{\sigma})
 \prod_{i=1}^n\sum_{y\in\mathcal{S}_{\sigma_i}}a^\varepsilon_{\mathfrak{m}_i,y}(\lambda)
 -\sum_{\boldsymbol{\sigma}\in S_n}\mathop{\mathrm{sgn}}(\boldsymbol{\sigma})\prod_{i=1}^n a^{\mathbf{Q}}_{i,\sigma_i}(\lambda) \Big|
  \leq\tilde{C} (\tilde C+|\lambda|)^{n-1}\varepsilon^\alpha.
\end{align}
\end{lem}
\begin{proof} 
First we need a generalization of the binomial formula that can be easily derived by a straightforward calculation.
Fix numbers $c_1,\ldots,c_n,d_1,\ldots,d_n\in\mathbb{R}$. Then
 \begin{align}\label{generalizedbinomial}
 \prod_{i=1}^n(c_i+d_i)=\sum_{i=0}^n G_i 
 \end{align}
 where
 \begin{align}
 G_0=\prod_{l=1}^n d_l\quad \text{and}\quad G_i=\sum_{1\leq j_1<\ldots<j_i\leq n}\prod_{k\in\left\{j_1,\ldots,j_i\right\}}c_k\prod_{\substack{l\in\left\{1,\ldots,n\right\}\\l\notin\left\{j_1,\ldots,j_i\right\}}}d_l.
 \end{align}

\noindent
Now, abbreviate the left hand side of \eqref{charpoly3} with $\text{(LHS)}$. We apply \eqref{generalizedbinomial} 
with $c_i=\sum_{y\in\mathcal{S}_{\sigma_i}}a^\varepsilon_{\mathfrak{m}_i,y}(\lambda)-a^\mathbf{Q}_{i,\sigma_i}(\lambda)$ and 
$d_i=a^\mathbf{Q}_{i,\sigma_i}(\lambda)$
to obtain
\begin{equation}
\label{charpoly4}
\begin{aligned}
 \text{(LHS)} 
 &\leq \sum_{\boldsymbol{\sigma}\in S_n}\Bigg|\prod_{i=1}^n\Bigg(\sum_{y\in\mathcal{S}_{\sigma_i}}a^\varepsilon_{\mathfrak{m}_i,y}(\lambda)-a^{\mathbf{Q}}_{i,\sigma_i}(\lambda)+a^{\mathbf{Q}}_{i,\sigma_i}(\lambda)\Bigg)-\prod_{i=1}^n a^{\mathbf{Q}}_{i,\sigma_i}(\lambda)\Bigg|\\
       &=    \sum_{\boldsymbol{\sigma}\in S_n}\Bigg|\prod_{i=1}^n c_i+\sum_{i=1}^{n-1}\sum_{1\leq j_1<\ldots<j_i\leq n}\prod_{k\in\left\{j_1,\ldots,j_i\right\}}c_k\prod_{\substack{l\in\left\{1,\ldots,n\right\}\\l\notin\left\{j_1,\ldots,j_i\right\}}}d_l\Bigg|\\
       &\leq \sum_{\boldsymbol{\sigma}\in S_n}\prod_{i=1}^n\Bigg|\sum_{y\in\mathcal{S}_{\sigma_i}}a^\varepsilon_{\mathfrak{m}_i,y}(\lambda)-a^{\mathbf{Q}}_{i,\sigma_i}(\lambda)\Bigg|+ \\
       &\quad\sum_{\boldsymbol{\sigma}\in S_n}\sum_{i=1}^{n-1}\sum_{1\leq j_1<\ldots<j_i\leq n}\underbrace{\prod_{k\in\left\{j_1,\ldots,j_i\right\}}\Bigg|\sum_{y\in\mathcal{S}_{\sigma_k}}a^\varepsilon_{\mathfrak{m}_k,y}(\lambda)-a^\mathbf{Q}_{k,\sigma_k}(\lambda)\Bigg|\prod_{\substack{l\in\left\{1,\ldots,n\right\}\\l\notin\{j_1,\ldots,j_i\}}}\Big| a^\mathbf{Q}_{l,\sigma_l}(\lambda)\Big|}_{(*)}.\\
\end{aligned}
\end{equation}
Since each row sum of $\mathbf{P}^\varepsilon-\mathbf{I}_{N}$ is equal to $0$, we have for every $x\in\mathcal{S}$
\begin{align}
 \sum_{y\in \mathcal{S}_i} (\mathbf{P}^\varepsilon-\mathbf{I}_N)_{x,y}
 =-\sum_{\substack{j=1\\j\neq i}}^n\sum_{y\in\mathcal{S}_j}(\mathbf{P}^\varepsilon-\mathbf{I}_{N})_{x,y}.
\end{align}
For $\boldsymbol{\sigma}\in S_n$ define $F(\boldsymbol{\sigma})\coloneqq \{1\leq i\leq n\colon \sigma_i=i\}$ 
as the set of all the fixed points of $\boldsymbol{\sigma}$.
Then, with the help of Lemma \ref{approxqij} we derive for some constant $C>0$
\begin{equation}
\begin{aligned}
 \sum_{\boldsymbol{\sigma}\in S_n}\prod_{i=1}^n\Bigg|
 \sum_{y\in\mathcal{S}_{\sigma_i}}a^\varepsilon_{\mathfrak{m}_i,y}(\lambda)-a^{\mathbf{Q}}_{i,\sigma_i}(\lambda)\Bigg|&\\
&\hspace*{-4.5cm}=\sum_{\boldsymbol{\sigma}\in S_n}
\Bigg[\prod_{\substack{i=1\\i\notin F(\boldsymbol{\sigma})}}^n
\Bigg|\sum_{y\in\mathcal{S}_{\sigma_i}}a^\varepsilon_{\mathfrak{m}_i,y}-a^{\mathbf{Q}}_{i,\sigma_i}\Bigg|\cdot
\prod_{\substack{i=1\\i\in F(\boldsymbol{\sigma})}}^n
\Bigg|-\sum_{\substack{j=1\\j\neq i}}^n\sum_{y\in\mathcal{S}_j}a^\varepsilon_{\mathfrak{m}_i,y}-\lambda-
\Bigg(-\sum_{\substack{j=1\\j\neq i}}^n q_{i,j}-\lambda\Bigg)\Bigg|\Bigg]\\
&\hspace*{-4.5cm}=\sum_{\boldsymbol{\sigma}\in S_n}\Bigg[\prod_{\substack{i=1\\i\notin F(\boldsymbol{\sigma})}}^n
\Bigg|\sum_{y\in\mathcal{S}_{\sigma_i}}\frac{1}{\varepsilon^\alpha h}p^\varepsilon_{\mathfrak{m}_i,y}-q_{i,\sigma_i}\Bigg|\cdot
\prod_{\substack{i=1\\i\in F(\boldsymbol{\sigma})}}^n
\Bigg|-\sum_{\substack{j=1\\j\neq i}}^n\sum_{y\in\mathcal{S}_j}\frac{1}{\varepsilon^\alpha h}p^\varepsilon_{\mathfrak{m}_i,y}
+\sum_{\substack{j=1\\j\neq i}}^n q_{i,j}\Bigg|\Bigg]\\
&\hspace*{-4.5cm}\leq	\sum_{\boldsymbol{\sigma}\in S_n}\left(C\varepsilon^\alpha\right)^{n-\vert F(\boldsymbol{\sigma})\vert}\left(nC\varepsilon^\alpha\right)^{\vert F(\boldsymbol{\sigma})\vert}\\
&\hspace*{-4.5cm}\leq	n!\left(nC\varepsilon^\alpha\right)^n.
\end{aligned}
\end{equation}

Similarly one can estimate the expression $(\ast)$ in \eqref{charpoly4} as follows:
\begin{equation}
\begin{aligned}
(\ast) &\leq \prod_{\substack{k\in\{j_1,\ldots,j_i\}\\k\notin F(\boldsymbol{\sigma)}}}
\Big|\sum_{y\in\mathcal{S}_k}\frac{1}{\varepsilon^\alpha h}p^\varepsilon_{\mathfrak{m}_k,y}-q_{k,\sigma_k}
\Big|\cdot 
\prod_{\substack{l\in \{1,\ldots,n \}\\l\notin\{j_1,\ldots,j_i\}\\l\notin F(\boldsymbol{\sigma})}} 
|q_{l,\sigma_l}|\\
    &\quad\cdot\prod_{\substack{k\in\{j_1,\ldots,j_i\}\\k\in F(\boldsymbol{\sigma)}}}
    \Big|-\sum_{\substack{s\neq k}}\sum_{y\in\mathcal{S}_s}\frac{1}{\varepsilon^\alpha h}p^\varepsilon_{\mathfrak{m}_k,y}
    +\sum_{s\neq k}q_{k,s}\Big|
    \cdot \prod_{\substack{l\in\{1,\ldots,n\}\\l\notin\{j_1,\ldots,j_i\}\\l\notin F(\boldsymbol{\sigma})}}
    | q_{l,l}-\lambda|\\
    &\leq(C\varepsilon^\alpha)^{ |\{k\in \{j_1,\ldots,j_i\}\colon\sigma_k\neq k|} 
    (nC\varepsilon^\alpha)^{|\{k\in\{j_1,\ldots,j_i\}\colon \sigma_k= k\}|}\\
    &\quad\times (\|\mathbf Q\|_\infty)^{\{k\in\{1,\ldots,n\}\setminus\{j_1,\ldots,j_i\}\colon \sigma_k\neq k\}}
    ( \|\mathbf Q\|_\infty+|\lambda|)^{\{k\in\{1,\ldots,n\}\setminus\{j_1,\ldots,j_i\}\colon\sigma_k= k\}}\\
    &\leq (nC\varepsilon^\alpha)^i( \|\mathbf Q\|_\infty+|\lambda|)^{n-i}.
\end{aligned}
\end{equation}
We finish the proof by continuing in (\ref{charpoly4}):
\begin{align}
\text{(LHS)} &\leq n!\cdot (nC\varepsilon^\alpha)^n
+n!\sum_{i=1}^{n-1} \binom{n}{i} (nC\varepsilon^\alpha)^i(\|\mathbf Q\|_\infty+|\lambda|)^{n-i}
\leq \tilde{C}\max_{1\leq k\leq n-1}( \|\mathbf Q\|_\infty+|\lambda|)^k\varepsilon^\alpha.
\end{align}
Assigning $\tilde C$ to be the maximum of $\tilde C$ and $1+\|\mathbf Q\|_\infty$ yields the result.
\end{proof}

We now complete the proof of Proposition \ref{proppolynome}.

\medskip

\noindent
\textit{Proof of Proposition \ref{proppolynome}.} 
First, we combine the equations \eqref{charpoly2} and \eqref{charpoly5} to obtain
\begin{equation}
\begin{aligned}
(-\varepsilon^\alpha h)^{N-n}P^\varepsilon(\lambda) 
&=\sum_{\boldsymbol{\sigma}\in S_n}\mathop{\mathrm{sgn}}(\boldsymbol{\sigma})
\prod_{i=1}^n\sum_{y\in\mathcal{S}_{\sigma_i}}a^\varepsilon_{\mathfrak{m}_i,y}\\
&\hspace*{-1cm}+(-1)^{N-n}\sum_{\boldsymbol{\pi}\in\Pi_0}\mathop{\mathrm{sgn}}(\boldsymbol{\pi})(-1)^{f_0(\boldsymbol{\pi})}
\Big(\sum_{l=0}^{N-n-1}\beta^\varepsilon_{l,0,\boldsymbol{\pi}}(\lambda) (\varepsilon^{\alpha}h)^{N-n-l}\Big)
\prod_{i=1}^na^\varepsilon_{\mathfrak{m}_i,\pi_{\mathfrak{m}_i}}(\lambda)\\
&\hspace*{-1cm}+(-1)^{N-n}\sum_{p=1}^{N-n}
\sum_{\boldsymbol{\pi}\in\Pi_p}\mathop{\mathrm{sgn}}(\boldsymbol{\pi})(-1)^{f_p(\boldsymbol{\pi})}
\Big(\sum_{l=0}^{N-n-p}\beta^\varepsilon_{l,p,\boldsymbol{\pi}}(\lambda) (\varepsilon^{\alpha}h)^{N-n-l}\Big)
\prod_{x\in\mathcal{S}^3(\boldsymbol{\pi})}a^\varepsilon_{x,\pi_x}.
\end{aligned}
\end{equation}
An application of Lemma \ref{abschpoly} then yields
\begin{equation}
\label{charpoly6}
\begin{aligned}
\Big|(-\varepsilon^\alpha h)^{N-n}P^\varepsilon(\lambda)-P^{\boldsymbol{Q}}(\lambda)\Big| 
 &	\leq \tilde{C}(\tilde C+|\lambda|)^{n-1}\varepsilon^\alpha\\
 &\hspace{-0cm} +\sum_{\boldsymbol{\pi}\in\Pi_0}\Big|\sum_{l=0}^{N-n-1}\beta^\varepsilon_{l,0,\boldsymbol{\pi}}(\lambda)
 (\varepsilon^{\alpha}h)^{N-n-l}\prod_{i=1}^n   a^\varepsilon_{\mathfrak{m}_i,\pi_{\mathfrak{m}_i}}(\lambda)\Big|\\
 &\hspace{-0cm} +\sum_{p=1}^{N-n}\sum_{\boldsymbol{\pi}\in\Pi_p}
 \Big|\sum_{l=0}^{N-n-p}\beta^\varepsilon_{l,p,\boldsymbol{\pi}}(\lambda)
(\varepsilon^{\alpha}h)^{N-n-l}\prod_{x\in\mathcal{S}^3(\boldsymbol{\pi})}a^\varepsilon_{x,\pi_x}\Big|.
\end{aligned}
\end{equation}
We have to estimate the second and third summand. 
For this let us use Lemma \ref{formelpermus} to renumber the states of $\mathcal{S}^1(\boldsymbol{\pi})$ and 
$\mathcal{S}^2(\boldsymbol{\pi})$
such that $\mathcal{S}^1(\boldsymbol{\pi})=\{x_1,\ldots,x_{f_p(\boldsymbol{\pi})}\}$ and 
$\mathcal{S}^2(\boldsymbol{\pi})=\{x_{f_p(\boldsymbol{\pi})+1},\ldots, x_{N-n-p}\}$. 
Then, from (\ref{betas}) it is easy to see that the numbers $\beta^\varepsilon_{l,p,\boldsymbol{\pi}}(\lambda)$ are
polynomials in $\lambda$ of degree $|\mathcal{S}^1(\boldsymbol{\pi})|=f_p(\boldsymbol{\pi})$. Expanding them with respect to 
$\lambda$ yields for $l=0$
\begin{align}
\beta^\varepsilon_{0,p,\boldsymbol{\pi}}(\lambda)
=\prod_{\nu=1}^{f_p(\boldsymbol{\pi})}(\lambda-\tilde{a}^\varepsilon_{x_\nu,x_\nu})\cdot 
\prod_{\nu=f_p(\boldsymbol{\pi})+1}^{N-n-p}(-\tilde{a}^\varepsilon_{x_\nu,x_\nu})
\end{align}
as well as for $1\leq l\leq N-n-p$
\begin{align}
 \beta^\varepsilon_{l,p,\boldsymbol{\pi}}(\lambda)
 =\sum_{1\leq \mu_1<\cdots<\mu_l\leq N-n-p}
 \Bigg(\prod_{\substack{\nu=1\\\nu\notin\left\{\mu_1,\ldots,\mu_l\right\}}}^{f_p(\boldsymbol{\pi})}(\lambda-a^\varepsilon_{x_\nu,x_\nu})\cdot
 \prod_{\substack{\nu=f_p(\boldsymbol{\pi})+1\\\nu\notin\left\{\mu_1,\ldots,\mu_l\right\}}}^{N-n-p}(-\tilde{a}^\varepsilon_{x_\nu,x_\nu})\Bigg).
\end{align}
Then, by Lemma \ref{ordnungeintraege} and Lemma \ref{l:proptransitionprobabilities} 
we can find a constant $\overline{c}=\overline{c}(\alpha,h,\delta,R)$ such that for
$0<\varepsilon<\varepsilon_0$
\begin{equation}
\begin{aligned}
|\beta^\varepsilon_{0,p,\boldsymbol{\pi}}(\lambda)|
=\prod_{\nu=1}^{f_p(\boldsymbol{\pi})}|\lambda-\tilde{a}^\varepsilon_{x_\nu,x_\nu}|\cdot
\prod_{\nu=f_p(\boldsymbol{\pi})+1}^{N-n-p}|\tilde{a}^\varepsilon_{x_\nu,x_\nu}|
&= \prod_{\nu=1}^{f_p(\boldsymbol{\pi})}
\Big|\lambda+\sum_{\substack{\kappa=1\\\kappa\neq \nu}}^{N}\tilde{a}^\varepsilon_{x_\kappa,x_\kappa}\Big|
\cdot \prod_{\nu=f_p(\boldsymbol{\pi})+1}^{N-n-p} |\tilde{a}^\varepsilon_{x_\nu,x_\nu}|\\
 &\leq\Big(\prod_{\nu=1}^{f_p(\boldsymbol{\pi})}(|\lambda|+N\overline{c})\Big)(\overline{c})^{N-n-p-f_p(\boldsymbol{\pi})}\\
 &\leq(|\lambda|+N\overline{c})^{f_p(\boldsymbol{\pi})}(\overline{c})^{N-n-p-f_p(\boldsymbol{\pi})}\\
 &\leq(|\lambda|+N\overline{c})^{N-n-p}
\end{aligned}
\end{equation}
and similarly
\begin{align}
|\beta^\varepsilon_{l,p,\boldsymbol{\pi}}(\lambda)|\leq N!(|\lambda|+N\overline{c})^{N-n-p-l},\quad 1\leq l\leq N-n-p.
\end{align}
Also, note that for $\boldsymbol{\pi}\in\Pi_p $ we have $\vert\mathcal{S}^3(\boldsymbol{\pi})\vert=n+p$. Then
\begin{align}
 \Big|\prod_{x\in\mathcal{S}^3(\boldsymbol{\pi})}a^\varepsilon_{x,\pi_x}(\lambda)\Big| &=
 \prod_{\substack{x\in\mathcal{S}^3(\boldsymbol{\pi})\\\pi_x=x}}|\lambda+\tilde{a}^\varepsilon_{x,x}|
 \prod_{\substack{x\in\mathcal{S}^3(\boldsymbol{\pi})\\\pi_x\neq x}}| a^\varepsilon_{x,\pi_x}|
 \leq (|\lambda|+ N\overline{c})^{n+p}.
\end{align}
We finish the proof by continuing in \eqref{charpoly6} for sufficiently small $\varepsilon$ and $C$ big enough
\begin{equation}
\begin{aligned}
 \Big|(-\varepsilon^\alpha h)^{N-n}P^\varepsilon(\lambda)-P^{\boldsymbol{Q}}(\lambda)\Big|\\
 &\hspace*{-3cm}\leq \tilde{C}(\tilde C+|\lambda|)^{n-1}\varepsilon^\alpha\\
 &\hspace*{-2cm} +(|\lambda|+ N\overline{c})^n\sum_{\boldsymbol{\pi}\in\Pi_0}
 \sum_{l=0}^{N-n-1}N!(|\lambda|+N\overline{c})^{N-n-l}(\varepsilon^{\alpha}h)^{N-n-l}\\
 &\hspace*{-2cm}+ \sum_{p=1}^{N-n}(|\lambda|+ N\overline{c})^{n+p}\sum_{\boldsymbol{\pi}\in\Pi_p}
 \sum_{l=0}^{N-n-p}N!(|\lambda|+N\overline{c})^{N-n-p-l}(\varepsilon^{\alpha}h)^{N-n-l}\\
 &\hspace*{-3cm}\leq \tilde{C}(\tilde C+|\lambda|)^{n-1}\varepsilon^\alpha\\
 &\hspace*{-2cm}+(|\lambda|+ N\overline{c})^n(N!)^2
 \Big[\sum_{l=0}^\infty (|\lambda|+ N\overline{c})^l (\varepsilon^{\alpha}h)^l-1\Big]\\
 &\hspace*{-2cm}+(|\lambda|+ N\overline{c})^n(N-n)(N!)^2
 \Big[\sum_{l=0}^\infty(|\lambda|+ N\overline{c})^l (\varepsilon^{\alpha}h)^l-1\Big]\\
 &\hspace*{-3cm}\leq \tilde{C}(\tilde C+|\lambda|)^{n-1}\varepsilon^\alpha 
 +2 (|\lambda|+ N\overline{c})^{n+1}(N!)^2\varepsilon^\alpha h
 +2(|\lambda|+ N\overline{c})^{n+1}(N-n)(N!)^2\varepsilon^\alpha h\\
 &\hspace*{-3cm}\leq C(C+|\lambda|)^{n+1}\e^\alpha. 
 \end{aligned}
\end{equation}
\hfill$\Box$

\section{Eigenvectors. Proof of Theorem \ref{Theoremeigenvektoren}}

For the investigation of the eigenvectors we will use some facts presented in \cite{Eckhoff00} 
where the author considered the spectrum of the unscaled matrix 
$\mathbf{I}_{N}-\mathbf{P}^\varepsilon$. 
Let  $\tilde{\lambda}_i=- \e^\alpha h\lambda_i^\e$, $1\leq i\leq N$, be the eigenvalues of this matrix. 
Of course, the eigenspaces of $\mathbf{Q}^\varepsilon$ and $\mathbf{I}_{N}-\mathbf{P}^\varepsilon$
coincide and
Theorem \ref{theoremeigenwerte} (i) and Proposition \ref{locationeigenvalues} imply that
\begin{align}\label{lambda}
 \tilde{\lambda}^\varepsilon_1=0,\quad\quad\tilde{\lambda}^{\varepsilon}_i=\mathcal O(\varepsilon^\alpha),\quad 2\leq i\leq n,
 \quad\text{and}\quad \operatorname{Re} \tilde{\lambda}^\varepsilon_i \geq 0,\quad  1\leq i\leq n.
\end{align}
The proof of Theorem \ref{Theoremeigenvektoren} needs some preparations. For a given subset $I\subset\mathcal{S}$ consider 
the first hitting time of $I$,
\begin{align}
 \sigma^\varepsilon_I\coloneqq \inf\{k\geq 0\colon Z^\varepsilon_k\in I\},
\end{align}
and, for convenience,  the first return time to $I$,
\begin{align}
 \tau^\varepsilon_I\coloneqq \inf \{k\geq 1\colon Z^\varepsilon_k\in I\}.
\end{align}
For $u\in\mathbb{C}$, $x\in\mathcal{S}$ and $I,J\subset\mathcal{S}$ consider the conditional Laplace transforms
\begin{align}
\label{laplaceG}
G_{I,J}^{x,\varepsilon}(u)\coloneqq 
\mathbb{E}_x\Big[\ex^{u\tau^\varepsilon_I}\mathbf{\mathop{1}}_{\{\tau^\varepsilon_I\leq \tau^\varepsilon_J\}}\Big]
\end{align}
and
\begin{align}
\label{laplaceK}
K_{I,J}^{x,\varepsilon}(u)\coloneqq 
\mathbb{E}_x\Big[\ex^{u\sigma^\varepsilon_I}\mathbf{\mathop{1}}_{\{\sigma^\varepsilon_I\leq \sigma^\varepsilon_J\}}\Big].
\end{align}
These expressions are finite for every $u\in\mathbb{C}$ such that $\operatorname{Re} u\leq 0$ but they can be infinite if the 
real part of $u$ is positive. We will discuss the
finiteness of these Laplace transforms for certain special cases later.

Since for any initial state $x\in\mathcal{S}\setminus I$ we have the equality $\sigma^\varepsilon_I=\tau^\varepsilon_I$, it follows that
\begin{align}\label{ZusammenhangGundK}
 K_{I,J}^{x,\varepsilon}(u)=\begin{cases}
                      G_{I,J}^{x,\varepsilon}(u) ,& x\notin I\cup J,\\
                      1                   ,& x\in I,\\
                      0		           ,& x\in J\setminus I.
                     \end{cases}
\end{align}
\noindent
In the following lemma we will prove another useful relation between these two functions (see Section 2 in \cite{Eckhoff00}). 
\begin{lem}\label{GleichungGK}
Fix an arbitrary $x\in\mathcal{S}$, subsets $I,J\subset\mathcal{S}$ and $u\in\mathbb{C}$ such that $G_{I,J}^{x,\varepsilon}(u)$ is finite. Then
\begin{align}
 \ex^u\sum_{y\in\mathcal{S}} p^\varepsilon_{x,y}K_{I,J}^{y,\varepsilon}(u)=G_{I,J}^{x,\varepsilon}(u).
\end{align}
\end{lem}
\begin{proof}
An application of the strong Markov property yields
\begin{equation}
\begin{aligned}
 G_{I,J}^{x,\varepsilon}(u)=\mathbb{E}_x\left[\ex^{u\tau^\varepsilon_I}\mathbf{\mathop{1}}_{\left\{\tau^\varepsilon_I\leq \tau^\varepsilon_J\right\}}\right]&=
 \sum_{y\in\mathcal{S}}\mathbb{E}_x\left[\ex^{u(1+\sigma^\varepsilon_I)}\mathbf{\mathop{1}}_{\left\{1+\sigma^\varepsilon_I\leq 1+ \sigma^\varepsilon_J\right\}}\vert Z^\varepsilon_1=y\right]\mathbb{P}_x(Z^\varepsilon_1=y)\\
 &=\ex^u\sum_{y\in\mathcal{S}} p^\varepsilon_{x,y}K_{I,J}^{y,\varepsilon}(u).
\end{aligned}
\end{equation}
\end{proof}

Now let us introduce the following quantities. Define $T^{\max}\in\mathbb N$ to be the maximal number of steps
the deterministic motion $Z^0$ needs to reach $\mathcal{M}$ when starting somewhere
in $\mathcal{S}$, 
\begin{align}
 T^{\max}=\max_{x\in\mathcal{S}}\min \{k\geq 1\colon Z^0_k(x)\in\mathcal{M} \}.
\end{align}
Recall the definitions of the interval boundaries $a_y$ and $b_y$ of $I_y$ and of the state $y^*(x)$ for a given $x\in\mathcal{S}$ and 
define
\begin{align}
\label{definitionD}
 D\coloneqq D(h,\delta)\coloneqq \min_{x\in\mathcal{S}} \operatorname{dist}\Big(x -h U'(x), [a_{y^*(x)}, b_{y^*(x)}]\Big).
\end{align}
This quantity gives the minimal distance which the deterministic motion has to the interval boundaries after one time step and is crucial when it comes to
comparing the trajectories of $Z^0$ and $Z^\varepsilon$. 
In particular, \eqref{estmatrix2} implies that
\begin{align}
\min_{x\in\mathcal S}\mathbb P_x (Z^0_1\neq  Z^\varepsilon_1)=1- \mathbb P_x (Z^0_1=  Z^\varepsilon_1)
=1-\mathbb P_x (Z^\varepsilon_1=y^*(x))
=p^\e_{x,y^*(x)}\geq \e^\alpha h D^{-\alpha}.
\end{align}
We denote $p^\varepsilon\coloneqq  \e^\alpha h D^{-\alpha}$, $p^\e=\mathcal O(\varepsilon^\alpha)$ as $\e\to 0$.

We are interested in the radius of convergence of the Laplace transforms $G_{I,J}^{x,\varepsilon}$ and $K_{I,J}^{x,\varepsilon}$ 
in dependence on $\varepsilon$ and for the case $I,J\subset\mathcal{M}$.
For this let
\begin{align}
 u^\varepsilon_\mathcal{M}\coloneqq 
 \sup\Big\{u\geq 0\colon G^{x,\varepsilon}_{\mathcal{M},\mathcal{M}}(u)=\mathbb{E}_x \ex^{u\tau^\varepsilon_\mathcal{M}}<\infty\quad 
 \text{for all }x\in\mathcal{S}\Big\}.
\end{align}
Obviously, for $u\geq 0 $ and $I\subset\mathcal{M}$ one always has $G_{I,\mathcal{M}}^{x,\varepsilon}(u)\leq G^{x,\varepsilon}_{\mathcal{M},\mathcal{M}}(u)$ so we can immediately conclude $G_{I,\mathcal{M}}^{x,\varepsilon}(u)<\infty$
for every $u\in\mathbb{C}$ with $0\leq \operatorname{Re} u< u^\varepsilon_\mathcal{M}$.

\begin{prp}\label{konvbereich}
 We have 
 \begin{align}
  \lim_{\varepsilon\rightarrow 0}u^\varepsilon_\mathcal{M}=\infty.
 \end{align}
\end{prp}

\begin{proof} 
Consider $k\geq 1$ such that  $T^{\max}+1\leq k\leq 2T^{\max}$. 
The occurrence of the event $\{\tau^\varepsilon_\mathcal{M}=k\}$ implies that the Markov chain $Z^\e$ 
did not follow the deterministic motion $Z^0$ at at least one step on the time interval $[0,k]$, namely  
\begin{align}
 \{\tau^\varepsilon_\mathcal{M}=k\}
 \subseteq \{Z^\e_l\neq Z^0_l \text{ for at least one }1\leq l\leq k\}.
\end{align}
Similarly, for $2T^{\max}+1\leq k\leq 3T^{\max}$ the occurrence of 
$\{\tau^\varepsilon_\mathcal{M}=k\}$ implies that such a deviation happened at least twice and so forth.
Hence, for $\mu\in\mathbb{N}$ and $\mu T^{\max}+1\leq k\leq (\mu+1)T^{\max}$ one has
\begin{align}
\label{bigjump}
\{\tau^\varepsilon_\mathcal{M}=k\}\subseteq \{Z^\e_l\neq Z^0_l \text{ for at least }\mu\text{ time instants } 
\{l_1,\ldots,l_\mu\}\subset \{1,\ldots,k \} \}.
\end{align}
Therefore for all
$x\in\mathcal{S}$, $\mu\in\mathbb{N}$ and $\mu T^{\max}+1\leq k\leq (\mu+1)T^{\max}$ we find with the help of the Markov property that
\begin{equation}
\begin{aligned}
\mathbb{P}_x(\tau^\varepsilon_\mathcal{M}=k)
&\leq \mathbb{P}_x \Big(Z^\e_l\neq Z^0_l \text{ for at least }\mu\text{ time instants } 
\{l_1,\ldots,l_\mu\}\subset \{1,\ldots,k \} \Big)\\
&\leq \mathbb{P}_x \Big (Z^\e_l\neq Z^0_l \text{ for at least }\mu\text{ time instants } 
\{l_1,\ldots,l_\mu \}\subset \{1,\ldots,(\mu +1)  T^{\max} \}  \Big)\\
&\leq
\sum_{l=\mu}^{(\mu+1)T^{\max}}   \binom{(\mu+1)T^{\max}}{l}  (p^\varepsilon)^l.    
\end{aligned}
\end{equation}
We apply the well known estimate for the binomial coefficient $\binom{n}{k}\leq (\frac{n \ex}{k})^k$, $n\geq 1$, $1\leq k\leq n$,
and the bound $\mu\leq l\leq (\mu+1)T^{\max}$ to get the estimate
\begin{align}
 \binom{(\mu+1)T^{\max}}{l}\leq C_1 K^\mu.
\end{align}
for some constants $C_1>0$ and $K>1$.
Then for some constant $C_2>0$:
\begin{align}
 \mathbb{P}_x(\tau^\varepsilon_\mathcal{M}=k)
 \leq C_1 \Big((\mu+1)T^{\max}-\mu+1\Big) K^{\mu} (p^\varepsilon)^\mu
 \leq C_2 \mu (K p^\varepsilon)^\mu.
\end{align}
For any $u\geq 0$ we set $C_3=\sum_{k=1}^{T^{\max}}\ex^{u k}\mathbb{P}_x(\tau^\varepsilon_\mathcal{M}=k)$ and obtain
\begin{equation}
\begin{aligned}
 \mathbb{E}_x \ex^{u\tau^\varepsilon_\mathcal{M}}& =\sum_{k=1}^\infty \ex^{u k}
 \mathbb{P}_x(\tau^\varepsilon_\mathcal{M}=k)\\
& =C_3+\sum_{\mu=1}^\infty\quad\sum_{k=\mu T^{\max}+1}^{(\mu+1)T^{\max}}\ex^{u k}\mathbb{P}_x(\tau^\varepsilon_\mathcal{M}=k)\\
&\leq C_3+C_4 \sum_{\mu=1}^\infty \ex^{u (\mu+1)T^{\max}} \mu (K p^\varepsilon)^\mu.
\end{aligned}
\end{equation}
This series converges if 
$\ex^{u T^{\max}} K p^{\varepsilon}<1$ or, equivalently, if 
$ u<\frac{1}{T^{\max}}\ln (\frac{1}{p^\varepsilon})-\ln(K)$.
The claim
of the proposition follows from the fact that $p^\varepsilon$ is of order $\mathcal O(\varepsilon^{\alpha})$.
\end{proof}

From this proposition it follows that for a fixed $u_0>0$ one can always choose $\varepsilon$ 
small enough such that $u_0\leq u^\varepsilon_\mathcal{M}$ and therefore we can immediately
deduce the following corollary
that will provide us with an approximation of $G_{I,J}^{x,\varepsilon}(u)$ with respect to $u$ in a neighbourhood of $0$.

\medskip
\noindent
\begin{cor}\label{corolaplacetaylor}
Fix $u_0>0$ and $\varepsilon_0>0$ such that $u_0\leq u^\varepsilon_\mathcal{M}$ and $0<\varepsilon<\varepsilon_0$. 
Let $x\in\mathcal{S}$ and $I\subseteq\mathcal{M}$.
Then there exists a constant $L>0$,
which can be chosen independently of $\varepsilon$ and $x$, 
such that for every $u\leq u_0$
\begin{align}
\left\vert G_{I,\mathcal{M}}^{x,\varepsilon}(u)-G_{I,\mathcal{M}}^{x,\varepsilon}(0)\right\vert\leq L\cdot \vert u\vert. 
\end{align}
\end{cor}

\begin{proof} The choice of $\varepsilon$ guarantees that $G_{I,\mathcal{M}}^{x,\varepsilon}(u)$ is finite for $u\leq u_0$. 
We then  use the Taylor expansion and obtain
\begin{align}
| G_{I,\mathcal{M}}^{x,\varepsilon}(u)-G_{I,\mathcal{M}}^{x,\varepsilon}(0)|
\leq\sum_{n=1}^\infty \frac{1}{n!}\Big|\frac{\di^n G_{I,\mathcal{M}}^{x,\varepsilon}(0)}{\di u^n}\Big|\cdot |u|^n.
\end{align}
We have to estimate the expressions $\frac{\di^n G_{I,\mathcal{M}}^{x,\varepsilon}(u)}{\di u^n}\Big|_{u=0}$ 
in the limit of small $\varepsilon$. An easy calculation yields
\begin{equation}
\begin{aligned}
\Big |\frac{\di^n G_{I,\mathcal{M}}^{x,\varepsilon}(0)}{\di u^n}\Big|
&=\sum_{k=1}^\infty \Big|\frac{\di^n}{\di u^n}\ex^{u k}\Big|_{u=0}\cdot \mathbb{P}_x(\tau^\varepsilon_I=k\leq\tau^\varepsilon_\mathcal{M}) \\
&=\sum_{k=1}^\infty k^n \mathbb{P}_x(\tau^\varepsilon_I=k\leq\tau^\varepsilon_\mathcal{M})\\
&\leq \sum_{k=1}^\infty k^n \mathbb{P}_x(k\leq\tau^\varepsilon_\mathcal{M})\\
 &=\sum_{k=1}^\infty k^n\sum_{l=k}^\infty\mathbb{P}_x(\tau^\varepsilon_\mathcal{M}=l).
\end{aligned}
\end{equation}
We proceed with similar arguments and the same constants $C_2$ and $K$ as in the proof of Proposition \ref{konvbereich}.
For a given $k\geq 1$ let $\mu_k\in\mathbb{N}$ be defined by the relation
$\mu_k T^{\max}+1\leq k\leq (\mu_k+1)T^{\max}$. Then,
\begin{equation}
\label{abschaetzungtaylor}
\begin{aligned}
\sum_{k=1}^\infty k^n\sum_{l=k}^\infty\mathbb{P}_x(\tau^\varepsilon_\mathcal{M}=l) 
& \leq \sum_{k=1}^\infty k^n\sum_{\mu=\mu_k}^\infty\sum_{l=\mu T^{\max}+1}^{(\mu+1) T^{\max}}\mathbb{P}_x(\tau^\varepsilon_\mathcal{M}=l)\\
& \leq \sum_{k=1}^\infty k^n\sum_{\mu=\mu_k}^\infty C_2\mu (K p^\varepsilon)^\mu\\
&\leq C_3 \sum_{k=1}^\infty k^n\sum_{\mu=\mu_k}^\infty (2K p^\varepsilon)^\mu\\
&=  \frac{C_3}{1-2K p^\varepsilon } \sum_{k=1}^\infty k^n  (2K p^\varepsilon)^{\mu_k} .
\end{aligned}
\end{equation}
We note that $\mu_k\leq \frac{k-1}{T^{\max}}$, abbreviate
$r^\varepsilon\coloneqq (2K p^\varepsilon)^{\frac{1}{T^{\max}}}$,
choose $\e$ small enough such that $r_\e\ex^{u_0}<1/2$ and find yet another constant $C_4>0$
such that
\begin{equation}
\begin{aligned}
 | G_{I,\mathcal{M}}^{x,\varepsilon}(u)-G_{I,\mathcal{M}}^{x,\varepsilon}(0)|
 &\leq C_4 \sum_{n=1}^\infty \frac{\vert u\vert^n}{n!}\sum_{k=0}^\infty k^n (r^\varepsilon)^k
 =C_4\sum_{k=0}^\infty  (r^\varepsilon)^k(\ex^{k\vert u\vert}-1)\\
 &=C_4\frac{r^\varepsilon(\ex^{\vert u\vert}-1)}{(1-r^\varepsilon)(1-r^\varepsilon\ex^{\vert u\vert})}\leq L \vert u\vert
\end{aligned}
\end{equation}
with a certain constant $L>0$.
\end{proof}

\begin{lem}
\label{l:eigenspace}
Let $\mathbf{A}$ be a $n\times n$ matrix such that $0\in\sigma(\mathbf{A})$ and 
eigenspace corresponding to the eigenvalue $0$ be one-dimensional. Let $\boldsymbol{\psi}$ be 
an eigenvector normalized such that 
$\|\boldsymbol{\psi}\|_\infty:=\max_{1\leq j\leq n}|\psi_j|=1$. Furthermore, let $\boldsymbol\rho(\e)$, $\e>0$, be a sequence of 
vectors such that $\lim_{\e\to 0}\|\boldsymbol\rho(\e)\|_\infty=0$ and assume that for every $\e>0$ the equation 
$\mathbf{A}\mathbf x=\boldsymbol{\rho}(\e)$ possesses a solution $\mathbf x=\boldsymbol{\psi}(\e)$ such that
$\|\boldsymbol{\psi}(\e)\|_\infty=1$. Then $\lim_{\e\to 0}\|\boldsymbol{\psi}(\e)-\boldsymbol{\psi}\|_\infty=0$.
\end{lem}
\begin{proof}
The matrix $\mathbf{A}$ defines a linear mapping from $\mathbb R^n$ to $\mathbb R^n$. Let $E_1 := \operatorname{ker}(\mathbf{A})$ 
and
$E_2 := \operatorname{im}(\mathbf{A})$ such that $\mathbb R^n=E_1\oplus E_2$.
Then it is easy to see that the restriction $\mathbf A_2:= \mathbf A|_{E_2}$ of $\mathbf A$ on $E_2$ a bijection from $E_2$ onto $E_2$. 
Indeed, since every $x \in \mathbb R^n$ such that $\mathbf A x = 0$ is an element of $E_1$ we
can deduce  that $\operatorname{ker}(\mathbf A_2) = {0}$ and hence $\mathbf A$ is injective. We also know that for every $x\in E_2$,
we can find $y = y_1 + y_2$, $y_1 \in E_1$, $y_2 \in E_2$ such that 
$x = \mathbf A y = \mathbf A y_2 = \mathbf A_2 y_2$ which
implies the surjectivity.

Now let $\boldsymbol \psi(\e)$ be the normalized solution to
$\mathbf A \boldsymbol \psi(\e)=\boldsymbol \rho(\e)$.
We then can write $\boldsymbol \psi(\e) = t(\e) + u(\e)$ for certain $t(\e) \in E_1$ and $u(\e)\in E_2$. Hence
$\mathbf A u(\e) =\boldsymbol \rho(\e)$.
Since $\mathbf A$ is invertible on $E_2$ and 
$(\mathbf A)^{-1}|_{E_2} = (\mathbf A_2)|_{E_2}$ 
we can write $u(\e) = (\mathbf A_2)^{-1}\boldsymbol \rho(\e)$
and therefore 
$\|\boldsymbol \psi(\e)-t(\e)\|_\infty = \|u(\e)\|_\infty \to 0$  as $\e\to 0$. 
Since $t(\e) \in E_1 = \{ a \boldsymbol \psi \colon a \in \mathbb R\}$ we can
conclude by the normalization condition on $\boldsymbol \psi(\e)$ that
$
\|\boldsymbol \psi (\e)-\boldsymbol \psi \|_\infty\to 0$
what finishes the proof.
\end{proof}

Now we are able to finish the proof of Theorem \ref{Theoremeigenvektoren}.

\medskip\noindent
\textit{Proof of Theorem \ref{Theoremeigenvektoren}:}

\noindent
For $i\leq 2\leq n$, 
let $u^{\varepsilon}_i$ be defined by $\tilde{\lambda}^{\varepsilon}_i=1-\ex^{-u^{\varepsilon}_i}$. By \eqref{lambda} it follows that 
$\lim_{\varepsilon\rightarrow 0}u^{\varepsilon}_i=0$.
We use a convenient representation of the eigenvectors of $\mathbf{I}_N-\mathbf{P}^\varepsilon$
obtained in \cite[Lemma 4.1]{Eckhoff00}. Namely, 
for any $\e>0$, applying this lemma with $I=\emptyset$, $J=\mathcal M$ we get that there
is a non-zero vector $(\psi^{\varepsilon,i}_{\mathfrak{m}_1},\ldots, \psi^{\varepsilon,i}_{\mathfrak{m}_n})$ such that the eigenvector of 
$\mathbf{I}_N-\mathbf{P}^\varepsilon$ is expressed in terms of the Laplace transforms $K^{x,\varepsilon}_{\mathfrak{m}_j,\mathcal{M}}$ as
\begin{align}\label{gleigenvector}
\psi^{\varepsilon,i}_x=\sum_{j=1}^n \psi^{\varepsilon,i}_{\mathfrak{m}_j} K^{x,\varepsilon}_{\mathfrak{m}_j,\mathcal{M}}(u^\varepsilon_i),
\quad x\in\mathcal{S}.
\end{align}  
For definiteness let us normalize $\max_{1\leq j\leq n}|\psi^{\varepsilon,i}_{\mathfrak{m}_j}|=1$.
Expanding the functions $u\mapsto K(u)$ with the help of Corollary \ref{corolaplacetaylor} we get
\begin{equation}
\begin{aligned} 
 K^{x,\varepsilon}_{\mathfrak{m}_{j},\mathcal{M}}(u^{\varepsilon}_i)
 =\begin{cases}
  1, \quad x= \mathfrak{m}_{j},\\
 \mathbb{P}_{x}\Big(\sigma^\varepsilon_{\mathfrak{m}_{j}}\leq \sigma^\varepsilon_\mathcal{M}\Big)+b_{x,\mathfrak{m}_j}(\e) 
 ,\quad x\in \Omega_j, \ 
 x\neq \mathfrak{m}_{j},\\ 
  \mathbb{P}_{x}\Big(\sigma^\varepsilon_{\mathfrak{m}_{k}}\leq \sigma^\varepsilon_\mathcal{M}\Big)+b_{x,\mathfrak{m}_k}(\e) , \quad x\in \Omega_k,\ 
  k\neq j,
  \end{cases}
\end{aligned} 
\end{equation}
 where $\max_{x,k}|b_{x,\mathfrak{m}_k}(\e)|=\mathcal O (\e^\alpha)$.
For $x\in\Omega_{j}$ we also have
\begin{align}
 \lim_{\varepsilon\rightarrow 0}\mathbb{P}_{x}\Big(\sigma^\varepsilon_{\mathfrak{m}_{j}}\leq \sigma^\varepsilon_\mathcal{M}\Big)= 1
 \quad\text{as well as}\quad\lim_{\varepsilon\rightarrow 0}
 \mathbb{P}_{x}\Big(\sigma^\varepsilon_{\mathfrak{m}_{k}}\leq \sigma^\varepsilon_\mathcal{M}\Big)= 0,
 \quad k\neq j.
\end{align}
%
Recall that $\boldsymbol{\psi}^{\varepsilon,i}$ satisfies the eigenvalue equation 
$\mathbf{A}^{\!\e}(\lambda_i^\e) \boldsymbol{\psi}^{\varepsilon,i}=0$.
From this system, let us single out $n$  equations indexed by $\mathfrak{m}_1,\ldots,\mathfrak{m}_n$:
\begin{equation}
\sum_{j=1}^n \psi^{\e,i}_{\mathfrak{m}_j} \sum_{y\in \mathcal S }(\mathbf{A}^{\!\e}(\lambda_i^\e))_{\mathfrak{m}_k,y}\cdot \Big(
\mathbb{P}_{y}\Big(\sigma^\varepsilon_{\mathfrak{m}_{j}}\leq \sigma^\varepsilon_\mathcal{M}\Big)+b_{y,\mathfrak{m}_j}(\e)\Big)
=0,\quad k=1,\ldots,n.
\end{equation}
Abbreviate $\mathbf{A}^\mathbf{Q}(\lambda):=\mathbf Q-\lambda \mathbf I_n$ and recall the following limits:
\begin{equation}
\begin{aligned} 
&\lim_{\e\to 0}\sum_{y\in \mathcal S_j }(\mathbf{A}^{\!\e}(\lambda))_{\mathfrak{m}_k,y}= (\mathbf{A}^\mathbf{Q}(\lambda))_{k,j},\ j=1,\ldots, n,\\
&\lim_{\e\to 0} \e^\alpha h\lambda_{i}^\e=\lambda_i^\mathbf{Q}.
\end{aligned} 
\end{equation}
Thus we can denote
\begin{equation}
\rho_k^i(\e)=\sum_{j=1}^n(\mathbf{A}^\mathbf{Q}(\lambda_i^\mathbf{Q}))_{k,j} \psi^{\e,i}_{\mathfrak{m}_j} ,\quad k=1,\ldots,n.
\end{equation}
where $\max_{k}|\rho^i_k(\e)|\to 0$ as $\e\to 0$. 
Recall that the eigenspace of $\lambda^\mathbf{Q}_i$ is one dimensional and apply Lemma \ref{l:eigenspace}.
That proves the assertion of the theorem.
\hfill$\Box$


\begin{thebibliography}{10}

\bibitem{Applebaum-09}
D.~Applebaum.
\newblock {\em {{L{\'e}}vy processes and stochastic calculus}}, volume 116 of
  {\em {Cambridge Studies in Advanced Mathematics}}.
\newblock Cambridge University Press, second edition, 2009.

\bibitem{BenziPSV-81}
R.~Benzi, G.~Parisi, A.~Sutera, and A.~Vulpiani.
\newblock {The mechanism of stochastic resonance}.
\newblock {\em Journal of Physics A}, 14:453--457, 1981.

\bibitem{Berglund-13}
N.~Berglund.
\newblock {Kramers{\rq} law: Validity, derivations and generalisations}.
\newblock {\em Markov Processes and Related Fields}, 19(3):459--490, 2013.

\bibitem{BerGen-10}
N.~Berglund and B.~Gentz.
\newblock {The {E}yring--{K}ramers law for potentials with nonquadratic
  saddles}.
\newblock {\em Markov Processes and Related Fields}, 16(3):549--598, 2010.

\bibitem{Berg1952}
H.~Bergstr{\"o}m.
\newblock {On some expansions of stable distribution functions}.
\newblock {\em Arkiv f{\"o}r Matematik}, 2(4):375--378, 1952.

\bibitem{BovEckGayKle02}
A.~Bovier, M.~Eckhoff, V.~Gayrard, and M.~Klein.
\newblock {Metastability and low lying spectra in reversible {M}arkov chains}.
\newblock {\em Communications in Mathematical Physics}, 228:219--255, 2002.

\bibitem{BovierEGK-04}
A.~Bovier, M.~Eckhoff, V.~Gayrard, and M.~Klein.
\newblock {Metastability in reversible diffusion processes I: Sharp asymptotics
  for capacities and exit times}.
\newblock {\em Journal of the European Mathematical Society}, 6(4):399--424,
  2004.

\bibitem{BovierGK-05}
A.~Bovier, V.~Gayrard, and M.~Klein.
\newblock {Metastability in reversible diffusion processes II: Precise
  asymptotics for small eigenvalues}.
\newblock {\em Journal of the European Mathematical Society}, 7(1):69--99,
  2005.

\bibitem{Burghoff14}
T.~Burghoff.
\newblock {\em {Spectral {P}roperties of a {D}iscrete L{\'e}vy-driven
  {M}etastable {S}ystem}}.
\newblock PhD thesis, Friedrich Schiller University Jena, 2014.

\bibitem{BuslovM-88}
V.~A. Buslov and K.~A. Makarov.
\newblock {Hierarchy of time scales in the case of weak diffusion}.
\newblock {\em Theor. Math. Phys.}, 76(2):818--826, 1988.

\bibitem{BuslovM-92}
V.~A. Buslov and K.~A. Makarov.
\newblock {Life times and lower eigenvalues of an operator of small diffusion}.
\newblock {\em Matematicheskie Zametki}, 51(1):20--31, 1992.

\bibitem{byl2009metastable}
K.~Byl and R.~Tedrake.
\newblock {Metastable walking machines}.
\newblock {\em The International Journal of Robotics Research},
  28(8):1040--1064, 2009.

\bibitem{Ca13}
M.~K. Cameron.
\newblock {Computing Freidlin{\rq}s cycles for the overdamped Langevin
  dynamics. Application to the Lennard--Jones-$38$ cluster}.
\newblock {\em Journal of Statistical Physics}, 152(3):493--518, 2013.

\bibitem{Cerrai-01}
S.~Cerrai.
\newblock {\em {Second order PDE's in finite and infinite dimension. A
  probabilistic approach}}, volume 1762 of {\em {Lecture notes in
  mathematics}}.
\newblock Springer, 2001.

\bibitem{ChechkinGKM-04}
A.~V. Chechkin, V.~Yu. Gonchar, J.~Klafter, R.~Metzler, and L.~V. Tanatarov.
\newblock {L{\'e}vy flights in a steep potential well}.
\newblock {\em Journal of Statistical Physics}, 115(5--6):1505--1535, 2004.

\bibitem{ChiangHS-87}
T.-S. Chiang, C.-R. Hwang, and S.-J. Sheu.
\newblock {Diffusion for global optimization in $\mathbb{R}^n$}.
\newblock {\em SIAM Journal on Control and Optimization}, 25(3):737--753, 1987.

\bibitem{Ditlevsen-99a}
P.~D. Ditlevsen.
\newblock {Observation of $\alpha$-stable noise induced millenial climate
  changes from an ice record}.
\newblock {\em Geophysical Research Letters}, 26(10):1441--1444, May 1999.

\bibitem{DubSpa07}
A.~Dubkov and B.~Spagnolo.
\newblock {Langevin approach to L{\'e}vy flights in fixed potentials: Exact
  results for stationary probability distributions}.
\newblock {\em Acta Physica Polonica B}, 38(5):1745--1758, 2007.

\bibitem{DybSokChe10}
B.~Dybiec, I.~M. Sokolov, and A.~V. Chechkin.
\newblock {Stationary states in single-well potentials under symmetric L{\'e}vy
  noises}.
\newblock {\em Journal of Statistical Mechanics: Theory and Experiment}, page
  P07008, 2010.

\bibitem{Eckhoff00}
M.~Eckhoff.
\newblock {\em {Capacity and low lying spectra of attractive {M}arkov chains}}.
\newblock PhD thesis, University of Potsdam, 2000.

\bibitem{FreidlinW-98}
M.~I. Freidlin and A.~D. Wentzell.
\newblock {\em {Random perturbations of dynamical systems}}, volume 260 of {\em
  {Grundlehren der Mathematischen Wissenschaften}}.
\newblock {Springer}, second edition, 1998.

\bibitem{Friedman-73}
A.~Friedman.
\newblock {The asymptotic behavior of the first real eigenvalue of a second
  order elliptic operator with a small parameter in the highest derivatives}.
\newblock {\em Indiana University Mathematics Journal}, 22(10):1005--1015,
  1973.

\bibitem{GalvesOV-87}
A.~Galves, E.~Olivieri, and M.~E. Vares.
\newblock {Metastability for a class of dynamical systems subject to small
  random perturbations}.
\newblock {\em The Annals of Probability}, 15(4):1288--1305, 1987.

\bibitem{HIPP2014}
S.~Herrmann, P.~Imkeller, I.~Pavlyukevich, and D.~Peithmann.
\newblock {\em {Stochastic Resonance: A Mathematical Approach in the Small
  Noise Limit}}, volume 194 of {\em {AMS Mathematical Surveys and Monographs}}.
\newblock American Mathematical Society, 2014.

\bibitem{HolleyKS-89}
R.~Holley, S.~Kusuoka, and D.~Stroock.
\newblock {Asymptotics of the spectral gap with applications to the theory of
  simulated annealing}.
\newblock {\em Journal of Functional Analysis}, 83(2):333--347, 1989.

\bibitem{HwangS-90}
C.-R. Hwang and S.-J. Sheu.
\newblock {Large-time behavior of perturbed diffusion {M}arkov processes with
  applications to the second eigenvalue problem for {F}okker-{P}lanck operators
  and simulated annealing}.
\newblock {\em Acta Applicandae Mathematicae}, 19(3):253--295, 1990.

\bibitem{ImkellerP-08}
P.~Imkeller and I.~Pavlyukevich.
\newblock {Metastable behaviour of small noise {L{\'e}}vy-driven diffusions}.
\newblock {\em ESAIM: Probaility and Statistics}, 12:412--437, 2008.

\bibitem{Khasm-59}
R.~Z. Khas$^\prime$minski.
\newblock {On positive solutions of the equation $\mathfrak{A}u+Vu=0$}.
\newblock {\em Theory of Probability and its Applications}, 4(3):309--318,
  1959.

\bibitem{KipnisN-85}
C.~Kipnis and C.~M. Newman.
\newblock {The metastable behavior of infrequently observed, weakly random,
  one-dimensional diffusion processes}.
\newblock {\em SIAM Journal on Applied Mathematics}, 45(6):972--982, 1985.

\bibitem{Kolokoltsov-00}
V.~N. Kolokoltsov.
\newblock {\em {Semiclassical analysis for diffusions and stochastic
  processes}}, volume 1724 of {\em {Lecture Notes in Mathematics}}.
\newblock {Springer}, 2000.

\bibitem{KolokoltsovM-96}
V.~N. Kolokol'tsov and K.~A. Makarov.
\newblock {Asymptotic spectral analysis of a small diffusion operator and the
  life times of the corresponding diffusion process}.
\newblock {\em Russian Journal of Mathematical Physics}, 4(3):341--360, 1996.

\bibitem{Kramers-40}
H.~A. Kramers.
\newblock {Brownian motion in a field of force and the diffusion model of
  chemical reactions}.
\newblock {\em Physica}, 7:284--304, 1940.

\bibitem{Kulik-09}
A.~M. Kulik.
\newblock {Exponential ergodicity of the solutions to {SDE}'s with a jump
  noise}.
\newblock {\em Stochastic Processes and their Applications}, 119(2):602--632,
  2009.

\bibitem{Makarov-85}
K.~A. Makarov.
\newblock {Division of the spectrum of an elliptic operator associated with
  ``small diffusion''}.
\newblock {\em {Vestnik Leningrad University. Mathematics}}, 18(1):27--36,
  1985.

\bibitem{MatkowskyS-81}
B.~J. Matkowsky and Z.~Schuss.
\newblock {Eigenvalues of the Fokker-Planck operator and the approach to
  equilibrium for diffusions in potential fields.}
\newblock {\em SIAM Journal on Applied Mathematics}, 40:242--254, 1981.

\bibitem{MetCheKla12}
R.~Metzler, A.~V. Chechkin, and J.~Klafter.
\newblock {L{\'e}vy statistics and anomalous transport: {L{\'e}}vy flights and
  subdiffusion}.
\newblock In M.~A. Meyers, editor, {\em {Computational Complexity. Theory,
  Techniques, and Applications}}, pages 1724--1745. Springer, 2012.

\bibitem{MetSchVdE-09}
P.~Metzner, C.~Sch{\"u}tte, and E.~Vanden-Eijnden.
\newblock {Transition path theory for Markov jump processes}.
\newblock {\em Multiscale Modeling \& Simulation}, 7(3):1192--1219, 2009.

\bibitem{Nicolis-82}
C.~Nicolis.
\newblock {Stochastic aspects of climatic transitions --- responses to periodic
  forcing}.
\newblock {\em Tellus}, 34:1--9, 1982.

\bibitem{OlivieriV-03}
E.~Olivieri and M.~E. Vares.
\newblock {\em {Large deviations and metastability}}.
\newblock {Encyclopedia of Mathematics and its Applications.} {Cambridge
  University Press}, 2003.

\bibitem{RahmSchm2002}
Q.~I. Rahman and G.~Schmeisser.
\newblock {\em {Analytic {T}heory of {P}olynomials}}.
\newblock {London Mathematical Society Monographs}. Clarendon Press, 2002.

\bibitem{SamorodnitskyG-03}
G.~Samorodnitsky and M.~Grigoriu.
\newblock {Tails of solutions of certain nonlinear stochastic differential
  equations driven by heavy tailed {L}{\'e}vy motions}.
\newblock {\em Stochastic Processes and their Applications}, 105(1):69--97,
  2003.

\bibitem{Schuss-10}
Z.~Schuss.
\newblock {\em {Theory and applications of stochastic processes. An analytical
  approach}}, volume 170 of {\em {Applied Mathematical Sciences}}.
\newblock Springer, 2010.

\bibitem{MatkowskyS-79}
Z.~Schuss and B.~J. Matkowsky.
\newblock {The exit problem: A new approach to diffusion across potential
  barriers}.
\newblock {\em SIAM Journal on Applied Mathematics}, 36(3):604--623, 1979.

\bibitem{schutte2013metastability}
C.~Sch{\"u}tte and M.~Sarich.
\newblock {\em {Metastability and Markov state models in molecular dynamics:
  modeling, analysis, algorithmic approaches}}, volume~24 of {\em {Courant
  Lecture Notes}}.
\newblock American Mathematical Society, 2013.

\bibitem{Tu-94}
P.~N.~V. Tu.
\newblock {\em {Dynamical systems: {A}n introduction with applications in
  {E}conomics and {B}iology}}.
\newblock Springer--Verlag, second edition, 1994.

\bibitem{varga2009}
R.~S. Varga.
\newblock {\em {Matrix {I}terative {A}nalysis}}.
\newblock {Springer Series in Computational Mathematics}. Springer, 2009.

\end{thebibliography}
%

\end{document}